\documentclass{amsart}
\usepackage{amssymb}
\usepackage{amsthm}
\usepackage{epsfig}
\usepackage{hyperref}
\usepackage{url}

\def\co{\colon\thinspace}
\newcommand{\bC}{\mathbb{C}}
\newcommand{\bE}{\mathbb{E}}
\newcommand{\bZ}{\mathbb{Z}}
\newcommand{\bN}{\mathbb{N}}
\newcommand{\bQ}{\mathbb{Q}}
\newcommand{\bH}{\mathbb{H}}
\newcommand{\bR}{\mathbb{R}}
\newcommand{\Qbar}{\overline{\bQ}}
\newcommand{\mc}{\mathcal}
\newcommand{\Hom}{\mathrm{Hom}}
\newcommand{\Out}{\mathrm{Out}}
\newcommand{\Aut}{\mathrm{Aut}}

\newcommand{\ab}{\mathrm{ab}}

\newcommand{\trace}{\mathrm{tr}}
\newcommand{\stker}{\underrightarrow{\mathrm{ker}}}

\newcommand{\cA}{\mathcal{A}}
\newcommand{\cB}{\mathcal{B}}
\newcommand{\cC}{\mathcal{C}}
\newcommand{\cD}{\mathcal{D}}
\newcommand{\cE}{\mathcal{E}}
\newcommand{\cF}{\mathcal{F}}
\newcommand{\cG}{\mathcal{G}}
\newcommand{\cH}{\mathcal{H}}

\newcommand{\cM}{\mathcal{M}}
\newcommand{\cN}{\mathcal{N}}

\newcommand{\cS}{\mathcal{S}}
\newcommand{\cU}{\mathcal{U}}
\newcommand{\cV}{\mathcal{V}}
\newcommand{\cW}{\mathcal{W}}
\newcommand{\cX}{\mathcal{X}}
\newcommand{\cY}{\mathcal{Y}}
\newcommand{\cSF}{\mathcal{SF}}
\newcommand{\cSol}{\mathcal{SOL}}

\newcommand{\llangle}{\langle\negthinspace\langle}
\newcommand{\rrangle}{\rangle\negthinspace\rangle}

\newtheorem{theorem}{Theorem}[section]
\newtheorem{lemma}[theorem]{Lemma}
\newtheorem{corollary}[theorem]{Corollary}
\newtheorem{proposition}[theorem]{Proposition}

\usepackage{amsthm}

\makeatletter
\newtheorem*{rep@theorem}{\rep@title}
\newcommand{\newreptheorem}[2]{%
\newenvironment{rep#1}[1]{%
 \def\rep@title{#2 \ref{##1}}%
 \begin{rep@theorem}}%
 {\end{rep@theorem}}}
\makeatother

\newreptheorem{theorem}{Theorem}
\newreptheorem{cor}{Corollary}

\theoremstyle{definition}
\newtheorem{remark}[theorem]{Remark}
\newtheorem{definition}[theorem]{Definition}
\newtheorem{example}[theorem]{Example}
\newtheorem{notation}[theorem]{Notation}

\def\into {\hookrightarrow}

\begin{document}

\title[Recognizing 3--manifold groups]{Recognizing geometric 3--manifold groups using the word problem}
\author{Daniel P.\ Groves}
\address{Department of Mathematics, Statistics, and Computer Science, University of Illinois at Chicago, 322 Science and Engineering Offices (M/C 249), 851 S. Morgan St., Chicago, IL 60607-7045, USA}
\email{groves@math.uic.edu}
\thanks{}

\author{Jason Fox Manning}
\address{244 Mathematics Building, Dept.\ of Mathematics, University at Buffalo, Buffalo, NY 14260-2900, USA}
\email{j399m@buffalo.edu}
\thanks{}

\author{Henry Wilton}
\address{Department of Mathematics, University College London, Gower Street, London, WC1E 6BT, UK}
\email{hwilton@math.ucl.ac.uk}
\thanks{The first author is supported by NSF Grant CAREER DMS-0953794.
  The second author is supported by NSF Grant DMS-1104703.  The third author is supported by an EPSRC Career Acceleration Fellowship.}



\date{}

\dedicatory{}

\begin{abstract}
Adyan and Rabin showed that most properties of groups cannot be algorithmically recognized from a finite presentation alone.  We prove that, if one is also given a solution to the word problem, then the class of fundamental groups of closed, geometric 3-manifolds is algorithmically recognizable.  In our terminology, the class of geometric 3-manifold groups is `recursive modulo the word problem'.   
\end{abstract}

\maketitle

\section{Introduction}

Decision problems have been central to the field of geometric group theory since its inception in the work of Dehn \cite{dehn_unendliche_1911}.  These may be split into {\em local} problems, where a group is given and questions are asked about elements or (finite) collections of elements, and {\em global} problems, where questions are asked about different finite presentations of groups.   Miller's survey article provides further details \cite{miller_iii_decision_1992}. 

We will focus on global problems: recognizing whether a finite group presentation defines a group with some particular property.  This is impossible for most natural properties of groups \cite{adyan_algorithmic_1955,adyan_unsolvability_1957,rabin_recursive_1958}; most notoriously, it is impossible to recognize whether or not a finite presentation presents the trivial group.

In order to aim towards positive recognition results, it is natural to consider a relative version of the problem.  We investigate certain `local-to-global' phenomena, where a local problem (specifically, the word problem) is assumed to be solvable and this solution is harnessed to determine global information.  Examples of these phenomena already occur in the literature.  In \cite{GrovesWilton09} it was proved that, given a solution to the word problem, it is possible to decide whether a finite presentation presents a free group, or a limit group; as mentioned in the same paper, surface groups can be recognized similarly (see Theorem \ref{thm: Surfaces rmwp} below).  
In \cite{manning:casson}, it was shown how to determine whether or not a closed $3$--manifold is hyperbolic, given access to a solution to the word problem in its fundamental group.
It is possible, with a solution to the word problem, to decide if a finitely presented group which has no $2$-torsion admits a nontrivial splitting over the trivial group \cite{Touikan09}.  On the other hand, amongst fundamental groups of non-positively curved square complexes (a class in which there is a uniform solution to the word problem), it is impossible to decide if a group has a nontrivial finite quotient \cite{bridson_profinite_2012}.

In the current paper and its sequel \cite{GMW2}, we turn our attention to the problem of detecting whether a finite presentation presents the fundamental group of a closed 3--manifold.  This is the mysterious dimension; in dimensions four and higher, \emph{every} finitely presentable group arises as the fundamental group of a closed manifold.   The main result of this paper is the following.

\begin{theorem}\label{thm: Woolly main theorem}
There exists an algorithm that, given a solution to the word problem, decides whether or not a finite presentation presents the fundamental group of a closed geometric 3--manifold.
\end{theorem}

In \cite{GMW2}, we remove the word `geometric' from the above theorem, and recognize 3--manifold groups using the word problem.

Note that we have been deliberately vague about what it means to be given a solution to the word problem in the statement of Theorem \ref{thm: Woolly main theorem}.  There are some subtleties to this notion, and we postpone a precise statement until Theorem \ref{t:Main} of the next section.  For now, the reader should note that Theorem \ref{thm: Woolly main theorem} implies that geometric 3--manifold groups can be recognized in many classes of groups that arise in practice, such as linear groups and automatic groups.  

In fact, we prove that that each of Thurston's eight geometries can individually be recognized using the word problem.  As a consequence, we recover the following well known corollary of geometrization (cf.\ \cite[Theorem 1.4]{manning:casson} and
\cite{LuoTillmannYang}). 

\begin{corollary} \label{Geometric?}
There is an algorithm that takes as input a triangulation of a closed 3--manifold $M$ and determines whether or not $M$ is geometric, and if so determines the geometry that $M$ admits.
\end{corollary}
\begin{proof}
The word problem for the fundamental group can be explicitly solved from a triangulation of a 3--manifold.
This follows from Perelman's Geometrization Theorem and a theorem of Hempel \cite{hempel87} (see Proposition \ref{hempelcor} below).
On the other hand, the geometry of a closed $3$--manifold is determined by its fundamental group.
\end{proof}

Along the way, we give a proof of the following theorem, which is of independent interest.   Although we presume this is known to the experts, we are not aware of a proof in the literature.  

\begin{reptheorem}{t: Virtual}
If $\Gamma$ is a torsion-free group with a subgroup of finite index isomorphic to $\pi_1M$, where $M$ is a closed, irreducible 3--manifold, then $\Gamma$ is also the fundamental group of a closed, irreducible 3--manifold.
\end{reptheorem}

The theorem is the dimension-three case of a conjecture of Wall in all dimensions \cite[p.\ 281]{farrell_borel_2002}. The result relies on Geometrization, Mostow Rigidity, the Convergence Group Theorem, and a theorem of Zimmermann.

\subsection*{Acknowledgements}
We thank Daryl Cooper for pointing out that Theorem \ref{t: Virtual} would
simplify our program.

\section{Recursive classes modulo the word problem}

We now turn to the formal setup which allows us to precisely state our main theorem (see Theorem \ref{t:Main} below).

Fix a countable alphabet, and let $\cU$ be the set of all finite
presentations on generators in that alphabet.  A subset $\cC \subset \cU$ is \emph{recursively enumerable} if there is an algorithm to recognize when a presentation $P$ is in $\cC$, and \emph{(absolutely) recursive} if both $\cC$ and its complement $\cU\smallsetminus\cC$ are recursively enumerable---that is, it can be algorithmically determined whether or not a finite presentation is an element of $\cC$.

We are concerned with isomorphism-closed classes of finitely presentable groups, which naturally correspond to isomorphism-closed subsets of $\cU$.  We will refer to such a subset as a \emph{class of groups}.  It is not hard to see that the class of \emph{perfect} groups is absolutely recursive.   However, this is an unusual phenomenon.  A class $\cC$ is {\em Markov} if there is a finitely presentable $G_0\in \cC$, and if there is a finitely presentable $H_0$ such that whenever $H_0$ is a subgroup of $G$, $G\notin\cC$. Note that the class consisting of just the trivial group is Markov.
A fundamental result of Adyan and Rabin states that if $\cC$ is Markov then $\cC$ is not recursive \cite{adyan_algorithmic_1955,adyan_unsolvability_1957,rabin_recursive_1958}.  In contrast, many natural classes of groups are recursively enumerable.

To remedy this problem, we restrict our attention to subsets $\cD$ of $\cU$ with better algorithmic properties.   The weakest natural assumption that one might impose is that the word problem should be uniformly solvable in $\cD$. Such classes of groups abound: linear groups, automatic groups, one-relator groups, residually finite groups etc.

\begin{remark}\label{rem: Trivial group}
If $\cD$ is a set of presentations in which the word problem is uniformly solvable, then there is an algorithm that recognizes whether \emph{or not} a finite presentation $P\in\cD$ represents the trivial group: the algorithm just checks whether or not every generator of $P$ is trivial.\end{remark}

This observation motivates our next definitions.  

\begin{definition}
A pair of disjoint sets of presentations $\mathcal{X}$ and $\mathcal{Y}$ is \emph{recursively separable} if there are recursively enumerable subsets $\mc{X}'$ and $\mc{Y}'$ of $\cU$ such that $\mc{X}\subseteq\mc{X}'$ but $\mc{Y}\cap\mc{X}'=\varnothing$ and $\mc{Y}\subseteq \mc{Y}'$ but $\mc{X}\cap\mc{Y}'=\varnothing$.
\end{definition}

\begin{definition}
Let $\cD$ be a set of finite presentations and let $\cC$ be a subset.  We will say that $\cC$ is \emph{recursive in $\cD$} if $\cC$ and $\cD\smallsetminus\cC$ are recursively separable.
\end{definition}

\begin{definition}
A class of groups $\cC$ is said to be \emph{recursive modulo the word problem} if, whenever $\cD$ is a set of presentations in which the word problem is uniformly solvable, the intersection $\cC\cap\cD$ is recursive in $\cD$.
\end{definition}

\begin{notation}
Throughout this paper $\cX_{\cC,\cD}$ and $\cY_{\cC,\cD}$ will always
denote recursively enumerable sets of presentations with the property
that $\cX_{\cC,\cD}$, $\cY_{\cC,\cD}$ separate $\cC\cap\cD$ and
$\cD\smallsetminus\cC$.  
\end{notation}
Note that there is neither a requirement that the union of $\cX_{\cC,\cD}$ and $\cY_{\cC,\cD}$ should be the whole of $\cU$, nor that their intersection should be trivial.  This is because these sets will come from a uniform solution to the word problem in $\cD$, which when applied to presentations not in $\cD$ may either fail to terminate or give incorrect answers.

\begin{remark}
  One could consider alternate definitions of ``recursive modulo the word problem''.  For example, say that a class of groups $\cC$ is \emph{recursive modulo a word-problem oracle} if both $\cC$ and its complement $\cU\smallsetminus \cC$ can be recursively enumerated by a Turing machine with an oracle which solves the word problem in every presentation in $\cU$.

  Since such an oracle is equivalent to an oracle which solves the halting problem, one can show that any recursively enumerable class of groups is recursive modulo a word-problem oracle.  The next example exhibits a class of groups which is recursively enumerable (hence recursive modulo a word-problem oracle) but not recursive
  modulo the word problem.
\end{remark}

\begin{example}
Let $\mc{FQ}$ be the class of all finitely presentable groups with a non-trivial finite quotient.  Then $\mc{FQ}$ is recursively enumerable, and hence both $\mc{FQ}$ and $\cU\smallsetminus\mc{FQ}$ can be enumerated by a Turing machine with an oracle for the halting problem.

On the other hand, in \cite{bridson_profinite_2012} it is shown that it is not possible to algorithmically determine whether or not the fundamental group of a compact non-positively curved cube complex is in $\mc{FQ}$.  As the word problem is uniformly solvable among such fundamental groups, it follows that $\mc{FQ}$ is not recursive modulo the word problem.
\end{example}

Let us denote by $\cM_n$ the class of fundamental groups of closed $n$-manifolds (so $\cM_n=\cU$ for $n\geq 4$) and by $\mc{AM}_n$ the class of fundamental groups of closed, aspherical $n$-manifolds.

\begin{example}
We shall see below that $\cM_1=\mc{AM}_1=\{\bZ\}$ is recursive modulo the word problem.
\end{example}

\begin{example}
In \cite{GrovesWilton09} it was shown that the class of finitely generated free groups is recursive modulo the word problem.  Similar arguments show that $\cM_2$ and $\mc{AM}_2$ are recursive modulo the word problem (see Section \ref{sec: Surface groups} below).
\end{example}

\begin{example}
It is a standard fact that $\cM_n=\cU$ for $n\geq 4$; in particular,
$\cM_n$ is recursive modulo the word problem when $n\geq 4$, for
trivial reasons.  Mark Sapir has proposed a proof that $\mc{AM}_n$ is not recursive modulo the word problem for $n\geq 4$ \cite{sapir_personal_2011}.
\end{example}

Therefore, there is a gap in our understanding in dimension three;
the purpose of this paper is to begin to begin to fill that gap.  Recall that a closed 3--manifold $M$ is \emph{geometric} if it admits a Riemannian metric such that the universal cover is homogeneous. 

\begin{theorem} \label{t:Main}
Let $\cG$ be the class of fundamental groups of closed geometric 3--manifolds.  Then $\cG$ is recursive modulo the word problem.
\end{theorem}

The reader is referred to \cite{scott:geom} for a comprehensive survey of geometric 3-manifolds.  In a subsequent paper \cite{GMW2}, we will remove the `geometric' hypothesis and
prove that $\cM_3$ and $\mc{AM}_3$ are both recursive modulo the word
problem.

By Thurston's classification, there are precisely eight mutually exclusive classes of closed geometric 3--manifold groups, as follows.

\begin{itemize}
\item $\cH$, the class of fundamental groups of closed hyperbolic 3--manifolds.
\item $\cSF_{+,\neq}$, the class of spherical 3--manifold groups.
\item $\cSF_{+,=}$, the class of fundamental groups of 3--manifolds modelled on $S^2\times\bR$.
\item $\cSF_{0,=}$, the class of Euclidean 3--manifold groups.
\item $\cSF_{0,\neq}$, the class of fundamental groups of manifolds modelled on the Heisenberg Lie group $\mathrm{Nil}$.
\item $\cSF_{-,=}$, the class of fundamental groups of manifolds modelled on $\bH^2\times\bR$.
\item $\cSF_{-,\neq}$, the class of fundamental groups of manifolds modelled on the Lie group $\widetilde{SL_2(\bR)}$.
\item $\cSol$, the class of fundamental groups of manifolds modelled on the Lie group $\mathrm{Sol}$. 
\end{itemize}

In the above notation, $\cSF_{\chi,e}$ is precisely the class of fundamental groups of closed Seifert fibred 3--manifolds such that the sign of the orbifold Euler characteristic of the base orbifold is given by $\chi$ and the Euler number of the Seifert bundle structure is either equal to zero or not according to $e$.

In our proof, the separate geometries are dealt with more or less separately, and so we in fact prove the following, stronger, result.

\begin{theorem}\label{thm:recursivegeometries}
Each of $\cH$, $\cSF_{\chi,e}$ (for all values of $\chi,e$) and $\cSol$ is recursive modulo the word problem.
\end{theorem}

\subsection{Outline}
Theorem \ref{thm:recursivegeometries} is the conjunction of the
statements in Theorems \ref{thm: SF+neq}, \ref{thm: SF+=}, \ref{thm:
  SF0=}, \ref{thm: SF0neq}, \ref{thm: SOL}, \ref{thm: SF-} and
\ref{thm: hyp}.

In Section \ref{sec:basic} we record some general facts
about recursiveness and recursiveness modulo the word problem of sets
of presentations.  In Section \ref{s:limit groups}, we develop some
aspects of the theory of limit groups which are useful for our
purposes.  In Section \ref{sec:wall3d}, we establish that torsion-free
virtual $3$--manifold groups are $3$--manifold groups.  In Section
\ref{sec: Surface groups}, we show that surface groups and aspherical
surface groups are recursive modulo the word problem.  

Sections
\ref{sec:spherical}--\ref{sec:hyp} establish, for each geometry in turn, 
the recursiveness modulo the word problem of 
the set of fundamental groups of closed $3$--manifolds admitting that
geometry.  Section \ref{sec:spherical} deals with spherical geometry,
Section \ref{sec:S2XR} with $S^2\times\bR$, Section \ref{sec:euclid} with
$\mathbb{E}^3$, Section \ref{sec:Nil} with $\mathrm{Nil}$, Section
\ref{sec:Sol} with $\mathrm{Sol}$, Section \ref{sec:SF-} with
$\mathbb{H}^2\times\bR$ and $\widetilde{SL_2(\bR)}$.  Finally,
Section \ref{sec:hyp} deals with hyperbolic geometry.

\section{Basic constructions and examples}\label{sec:basic}

\subsection{First properties}

\begin{remark}\label{rem: Unions and intersections}
If $\mc{X}$ and $\mc{Y}$ are recursively enumerable then so are $\mc{X}\cap\mc{Y}$ and $\mc{X}\cup\mc{Y}$.
\end{remark}

\begin{lemma}\label{lem: Recursive subsets of recursive subsets}
Let $\cC\subseteq\cD\subseteq\mc{E}$ be sets of presentations.
\begin{enumerate}
\item\label{basic1} If $\mc{C}$ is recursive in $\mc{D}$ and $\mc{D}$ is recursive in $\mc{E}$ then $\mc{C}$ is recursive in $\mc{E}$.
\item\label{basic2} If $\cC$ is recursive in $\mc{E}$ then $\cC$ is also recursive in $\cD$.
\end{enumerate}
\end{lemma}
\begin{proof}
For the first assertion, set $\mc{X}_{\cC,\mc{E}}=\cX_{\cC,\cD}\cap\cX_{\cD,\mc{E}}$ and $\mc{Y}_{\cC,\mc{E}}=\cY_{\cC,\cD}\cup\cY_{\cD,\mc{E}}$.  For the second, just note that $\cX_{\cC,\cD}=\cX_{\cC,\mc{E}}$ and $\cY_{\cC,\cD}=\cY_{\cC,\mc{E}}$ are as required.
\end{proof}

\begin{lemma}\label{lem: recursive subsets 2}
If $\cD$ is recursive modulo the word problem and $\cC$ is recursive in $\cD$ then $\cC$ is also recursive modulo the word problem.
\end{lemma}
\begin{proof}
Suppose the word problem is uniformly solvable in $\cE$.  Taking
\[
\cX_{\cC,\cE}=\cX_{\cC,\cD}\cap\cX_{\cD,\cE}
\]
and
\[
\cY_{\cC,\cE}=\cY_{\cC,\cD}\cup\cY_{\cD,\cE}
\]
proves the result.
\end{proof}

\subsection{The Isomorphism Problem}

\begin{remark}
Let $\cD$ be a class of groups in which the isomorphism problem is solvable.  Then every singleton sub-class-of-groups is recursive in $\cD$.

The converse to this is not quite true.  Saying that the isomorphism problem is solvable in $\cD$ is the same as saying that the singleton classes of groups in $\cD$ are \emph{uniformly} recursive, in the sense that there is a single Turing machine that can recognize the class of any group.
\end{remark}

\subsection{The Reidemeister--Schreier Algorithm}

The Reidemeister--Schreier Algorithm enumerates the finite-index subgroups of a given finitely presented group (see \cite[Proposition 4.1]{lyndon_combinatorial_1977}).  Topologically, the Reidemeister--Schreier Algorithm can be thought of as just listing all possible covering spaces of a finite CW complex of index at most $k$.

\begin{theorem}[Reidemeister--Schreier]
There is an algorithm that takes as input a presentation $P\in\cU$ and a natural number $k$ and outputs a list of presentations of all subgroups of index at most $k$ together with embeddings into the group presented by $P$.
\end{theorem}

We shall denote this output set of presentations by $\cS_k\{P\}$.  More generally, for any set of presentations $\cC$, let
\[
\cS_k\cC=\bigcup_{P\in\cC}\cS_k\{P\}
\] 
and let
\[
\cS\cC=\bigcup_{k\in\bN}\cS_k\cC~,
\]
the class of all finite index subgroups of groups in $\cC$.  Furthermore, let
\[
\cV_k\cC=\{P\in\cU\mid \cS_k\{P\}\cap\cC\neq\varnothing\}
\]
for each $k\in\bN$ and let
\[
\cV\cC=\bigcup_{k\in\bN}\cV_k\cC~,
\]
the class of all virtually-$\cC$ groups.
Finally, we use the following notation for the class of groups all of whose
bounded index subgroups lie in $\cC$:
\[
\cW_k\cC=\{P\in\cU\mid \cS_k\{P\}\subseteq\cC\}
\]
for each $k\in\bN$.  Finally,
\[
\cW\cC=\bigcap_{k\in\bN}\cW_k\cC~
\]
denotes the groups all of whose finite index subgroups lie in $\cC$.

It is sometimes useful to be able to allow the index that we consider to vary with the presentation.  Suppose therefore that $f:\cU\to\bN$ is a computable function.  We will write
\[
\cS_f\cC=\bigcup_{P\in\cC}\cS_{f(P)}\{P\}
\] 
as well as
\[
\cV_f\cC=\{P\in\cU\mid \cS_{f(P)}\{P\}\cap\cC\neq\varnothing\}
\]
and
\[
\cW_f\cC=\{P\in\cU\mid \cS_{f(P)}\{P\}\subseteq\cC\}
\]

The following is an immediate application of the Reidemeister--Schreier Algorithm.
\begin{lemma}
Let $\cC$ be a set of presentations and $f:\cU\to\bN$ a computable function.
\begin{enumerate}
\item If $\cC$ is recursively enumerable then so are $\cS_f\cC$, $\cV_f\cC$ and $\cW_f\cC$, and hence also  $\cS\cC$ and $\cV\cC$.
\item If the word problem is uniformly solvable in $\cC$ then the word problem is also uniformly solvable in $\cS_f\cC$, $\cV_f\cC$ and $\cW_f\cC$, and hence also in $\cS\cC$ and $\cV\cC$.
\end{enumerate}
\end{lemma}

The following, extremely useful, result follows.

\begin{lemma}\label{lem: Virtually recursive modulo the word problem}
Let $f:\cU\to\bN$ be a computable function.  If $\cC$ is recursive modulo the word problem then so is $\cV_f\cC$.
\end{lemma}
\begin{proof}
Suppose that the word problem is uniformly solvable in $\cD$.  Then it is also uniformly solvable in $\cS_f\cD$, and so $\cC\cap \cS_f\cD$ is recursive in $\cS_f\cD$.  The sets $\cV_f\cX_{\cC,\cS_f\cD}$ and $\cW_f\cY_{\cC,\cS_f\cD}$ are both recursively enumerable.  Because  $\cD\subseteq\cW_f\cS_f\cD$, it follows that
\[
\cV_f\cC\cap\cD\subseteq\cV_f\cC\cap\cW_f\cS_f\cD\subseteq \cV_f(\cC\cap\cS_f\cD) \subseteq \cV_f\cX_{\cC,\cS_f\cD}
\]
and also that
\[
\cD\smallsetminus \cV_f\cC\subseteq \cW_f\cS_f\cD\smallsetminus \cV_f\cC\subseteq \cW_f(\cS_f\cD\smallsetminus \cC)\subseteq\cW_f\cY_{\cC,\cS_f\cD}~.
\]
Furthermore,
\begin{eqnarray*}
(\cV_f\cC\cap\cD)\cap\cW_f\cY_{\cC,\cS_f\cD}&\subseteq & \cV_f\cC\cap\cW_f\cS_f\cD\cap\cW_f\cY_{\cC,\cS_f\cD}\\
&\subseteq &\cV_f(\cC\cap\cS_f\cD\cap\cY_{\cC,\cS_f\cD})\\
&=&\varnothing
\end{eqnarray*}
and
\begin{eqnarray*}
(\cD\smallsetminus \cV_f\cC)\cap\cV_f\cX_{\cC,\cS_f\cD}&\subseteq & (\cW_f\cS_f\cD\smallsetminus \cV_f\cC)\cap\cV_f\cX_{\cC,\cS_f\cD}\\
&\subseteq& \cV_f((\cS_f\cD\smallsetminus \cC)\cap\cX_{\cC,\cS_f\cD})\\
&=&\varnothing~.
\end{eqnarray*}
Therefore, setting $\cX_{\cC,\cV_f\cC}=\cV_f\cX_{\cC,\cS_f\cD}$ and $\cY_{\cC,\cV_f\cC}=\cW_f\cY_{\cC,\cS_f\cD}$ proves the lemma.
\end{proof}

\section{Limit groups}\label{s:limit groups}

In this section, we will develop the theory of limit groups in a form that is convenient for our purposes.  We make no claims to originality.  Most of the results of this section already appear in the literature; see \cite{sela_diophantine_2001}, \cite{GrovesWilton09} and the references therein, and also \cite{kharlampovich_irreducible_1998} \emph{et seq.}

\subsection{Definitions}

\begin{definition}
Let $G$ be a group.  A sequence of homomorphisms $(f_n:G\to\Gamma_n)$ is called \emph{convergent} if, for every $g\in G$, either $f_n(g)=1$ for all sufficiently large $n$ or $f_n(g)\neq 1$ for all sufficiently large $n$. 
\end{definition}

\begin{remark}
\ 
\begin{enumerate}
\item Some authors call such a sequence \emph{stable}.  We, however, prefer the term \emph{convergent}, which agrees with the fact this definition corresponds to convergence in the Gromov--Grigorchuk topology on marked groups.
\item An easy diagonal argument shows that if $G$ is countable then every sequence of homomorphisms has a convergent subsequence.
\end{enumerate}
\end{remark}

\begin{definition}
Given a convergent sequence $f_n$, the \emph{stable kernel} $\stker~f_n$ is the normal subgroup of all $g\in G$ with $f_n(g)=1$ for all sufficiently large $n$.  A group $L$ is called a \emph{$\cC$-limit group} if there is a finitely generated group $G$ and a convergent sequence of homomorphisms $f_n:G\to\Gamma_n$ with each $\Gamma_n\in\cC$ such that $L\cong G/\stker~f_n$.  We will denote the class of $\cC$-limit groups by $\overline{\cC}$.
\end{definition}

\begin{example}
If $H$ is a finitely generated subgroup of a group $\Gamma\in\cC$ then $H$ is a $\cC$-limit group.
\end{example}

\begin{example}
Let $\cN_k$ be the class of nilpotent groups of class at most $k$.  Then $\overline{\cN}_k=\cN_k$.
\end{example}

We are particularly interested in the class $\cF$ of free non-abelian groups.

\begin{example}
If $\Sigma$ is a compact, orientable surface then $\pi_1\Sigma\in\overline{\cF}$.
\end{example}

\begin{definition}
A \emph{$\cC$-factor set} for a group $G$ is a finite set of epimorphisms
\[
\{\eta_i:G\to L_i\}
\]
such that each $L_i\in\overline{\cC}$ and such that every homomorphism
$f$ from $G$ to an element of $\cC$ factors as $f=f'\circ \eta_i$ for some $i$ and some epimorphism $f':L_i\to \Gamma$ for some $\Gamma\in\cC$.    A factor set of minimal cardinality is called \emph{minimal}.
\end{definition}

\begin{example}
If $\Gamma$ is a $\cC$-limit group then $\{\Gamma\to\Gamma\}$ is a factor set.
\end{example}

\begin{lemma}
Let $G$ be any group. If there is a $\cC$-factor set for $G$ then the minimal $\cC$-factor set is unique.
\end{lemma}
\begin{proof}
Let $\{\eta_i\}$  and $\{\xi_j\}$ be minimal $\cC$-factor sets.  Then
every $\xi_j$ factors through some $\eta_{i(j)}$ and, likewise, every
$\eta_i$ factors through some $\xi_{j(i)}$.  Minimality now implies
that the maps $i\mapsto j(i)$ and $j\mapsto i(j)$ are mutual inverses, and the result follows.
\end{proof}

\begin{definition}
If there is a $\cC$-factor set for every finitely presentable $G$ then we say that \emph{factor sets exist over $\cC$}.
\end{definition}

The condition that every $\cC$-limit group is finitely presentable is not particularly natural; nevertheless, it happens to hold in the examples that we are interested in, and this makes the theory easier.

\begin{lemma}\label{fp_then_factorsets}
If every $\cC$-limit group is finitely presentable then factor sets exist over $\cC$.
\end{lemma}
\begin{proof}
Suppose $G$ is a finitely presented group without a factor set.  Let
$\{\eta_n:G\to\Gamma_n\}$ be the (countable) set of all $\cC$-limit
group quotients of $G$.  For each natural number $n$, let
$f_n:G\to\Gamma_n$ be a homomorphism that does not factor through
$\{\eta_1,\ldots,\eta_n\}$.  Now pass to an infinite convergent
subsequence and let $L$ be the limit group $G/\stker f_n$.  Because
$L$ is finitely presented, $L=G/\langle\langle
r_1,\ldots,r_k\rangle\rangle$ for some finite $k$ and elements $r_i\in
\stker f_n$.  By definition, there is an $n_0$ such that every
$r_i\in\ker f_n$ for all $n\geq n_0$; in particular, $f_n$ factors
through the natural map $\eta:G\to L$ for all $n\geq n_0$.  But $\eta
= \eta_m$ for some $m$, contradicting the choice of the sequence $\{f_n\}$.
\end{proof}

\subsection{Decision problems for $\cC$-limit groups}

\begin{definition}
We say that \emph{factor sets are computable over $\cC$} if there is an algorithm that takes as input a finite presentation $\langle X\mid R\rangle$ for a group $G$ and outputs a finite list of presentations $F_{\cC}(G)=\{L_i\cong \langle X\mid R\cup S_i\rangle\}$ such that set of natural maps $\{G\to L_i\}$ is a minimal factor set for $G$.
\end{definition}

Note that if factor sets are computable over $\cC$ then, in particular, every $\cC$-limit group is finitely presentable.

\begin{lemma}\label{lem: Computable factor sets implies RE closure}
If factor sets are computable over $\cC$ then $\overline{\cC}$ is recursively enumerable.
\end{lemma}
\begin{proof}
A Turing machine that enumerates $\cU$ and computes minimal factor sets will output a list of elements of $\overline{\cC}$.  Every element of $\overline{\cC}$ appears in its own minimal factor set, so the list is complete.
\end{proof}

\begin{theorem}\label{thm: Computable factors sets implies rmwp}
If factor sets are computable over $\cC$ then $\overline{\cC}$ is recursive modulo the word problem.
\end{theorem}
\begin{proof}
Suppose the word problem is uniformly solvable in a set of finite presentations $\cD$.  By the previous lemma, we can set $\cX_{\overline{\cC},\cD}=\overline{\cC}$.  To define $\cY_{\overline{\cC},\cD}$, we will describe a Turing machine that confirms if an element of $\cD$ does not present a $\cC$-limit group.  For a group $G$ presented by an element of $\cD$, this Turing machine computes a minimal $\cC$-factor set.  If this factor set has more than one element then $G\notin\overline{\cC}$.  Otherwise, the machine outputs a single epimorphism $\eta:G\to L$ where $L\in\overline{\cC}$.  The machine the determines whether the tautological section $\sigma:L\to G$ is a homomorphism, using the solution to the word problem in $G$.  The group $G$ is a $\cC$-limit group if and only the section is a homomorphism.
\end{proof}

\subsection{Abelian and nilpotent groups}

Let $\cA$ be the class of finitely presented abelian groups.  For
$k\geq 1$, let $\cN_k$ be the class of finitely presented nilpotent
groups of class at most $k$.

\begin{example}
Because every map from a group $G$ to an abelian group factors through the abelianization $G_\ab$, a factor set for $G$ just consists of the canonical map $G\to G_\ab$.  In particular, factor sets are computable and $\overline{\cA}=\cA$.
\end{example}

\begin{example}
Let $G$ be a group.  Recall the {\em lower central series} of $G$ is defined by $\gamma_1(G) = G$ and $\gamma_{i+1}(G) = [\gamma_i(G),G]$.  A group $G$ is nilpotent of class at most $k$ if and only if $\gamma_{k+1}(G) = \{ 1 \}$.  Moreover, for any group $G$, any map from $G$ to a group which is nilpotent of class at most $k$ factors through the canonical map $G \to G / \gamma_{k+1}(G)$.  If $G$ is finitely presented then so is $G_k = G/\gamma_{k+1}(G)$.  A presentation for $G_k$ is computed by the ANU NQ algorithm \cite{Nickel}.  Therefore, factor sets are computable for $\cN_k$ and $\overline{\cN}_k = \cN_k$.
\end{example}

Applying Theorem \ref{thm: Computable factors sets implies rmwp}, we immediately obtain that nilpotent groups of fixed class are recursive modulo the word problem.

\begin{theorem}
For any $k$, $\cN_k$ is recursive modulo the word problem. In particular, $\cA$, the class of finitely generated abelian groups, is recursive modulo the word problem.
\end{theorem}

\begin{corollary}
The class $\{\bZ\}=\cM_1=\mc{AM}_1$ is recursive modulo the word problem.
\end{corollary}
\begin{proof}
The isomorphism problem for finitely generated abelian groups is solvable, and hence $\{\bZ\}$ is recursive in $\cA$.
\end{proof}
\begin{corollary}\label{abeliansingleton}
  Fix a finitely generated abelian group $A$.  Then $\{A\}$ is
  recursive modulo the word problem.
\end{corollary}

\subsection{The universal theory}

Over the class of free groups, it is more difficult to compute factor sets.  We will therefore need a different set of criteria.

\begin{theorem}\label{thm: Decidable ut and re limit groups implies computable factor sets}
Let $\cC$ be a class of finitely presented groups. Suppose that $\overline{\cC}$ is recursively enumerable and that the universal theory of $\cC$ is decidable.  Then factor sets are computable over $\cC$.
\end{theorem}
\begin{proof}
Since $\overline{\cC}$ is recursively enumerable, it contains only
finitely presentable groups.  Hence by Lemma \ref{fp_then_factorsets}
factor sets exist over $\mc{C}$.

Let $G\in\cU$.  We first prove that it is possible to compute some factor set from a presentation $\langle X\mid R\rangle$ for $G$.  Because $\overline{\cC}$ is recursively enumerable, one can enumerate finite sets of presentations for $\cC$-limit groups $\{L_i=\langle X\mid R\cup S_i\rangle\}$; note that such presentations come with a natural epimorphism $\eta_i$ from $G$.  The statement that every homomorphism from $G$ to $\Gamma\in\cC$ factors through some $\eta_i$ is equivalent to the assertion that the sentence
\[
\forall X,~ R =1\Rightarrow \bigvee_i S_i=1
\]
holds in $\Gamma$.  Using the decidability of the universal theory, one can determine whether or not a finite set is a factor set.  As factor sets always exist, a na\"ive search will eventually find one.  Therefore, it is possible to compute a factor set.

We next observe that it is possible to determine whether or not a factor set is minimal.  Indeed, if it is not, then some $\eta_j$ factors through some  $\eta_i$, which is equivalent to the sentence
\[
\forall X,~R=1 \wedge S_i=1 \Rightarrow S_j=1~.
\]
Therefore, a given finite set is minimal if and only if the sentence
\[
\forall X,~\bigwedge _{i\neq j}R=1 \wedge S_i=1 \nRightarrow S_j=1
\]
holds in every $\Gamma\in\cC$.  Again, this can be determined using the decidability of the universal theory, and so a minimal factor set is computable.
\end{proof}

The next corollary follows immediately from the combination of Theorems \ref{thm: Computable factors sets implies rmwp} and \ref{thm: Decidable ut and re limit groups implies computable factor sets}.

\begin{corollary}\label{cor: Limit group criterion for rmwp}
If $\overline{\cC}$ is recursively enumerable and the universal theory of $\cC$ is decidable then $\overline{\cC}$ is recursive modulo the word problem.
\end{corollary}

\subsection{Limit groups over free groups}

In this section, we summarize known results when $\cC=\cF$, the class of free groups.  The set of $\cF$-limit groups is precisely the set of {\em limit groups} considered by Sela, as well as by Kharlampovich--Miasnikov under the name \emph{(finitely generated) fully residually free groups}. 

The following results show that the results of the previous section can be applied to $\cF$.  Makanin proved that the universal theory of free groups is decidable \cite{makanin_decidability_1984}.

\begin{theorem}\label{thm: Makanin's theorem}
The universal theory of $\cF$ is decidable.
\end{theorem}

Sela showed that $\cF$-limit groups are finitely presentable \cite{sela_diophantine_2001}; Guirardel gave another proof \cite{guirardel_limit_2004}.

\begin{theorem}\label{thm: Limit groups are fp}
Every group in $\overline{\cF}$ is finitely presentable.
\end{theorem}

The first and third authors proved that the class of $\cF$-limit groups is recursively enumerable \cite{GrovesWilton09}.

\begin{theorem}\label{thm: Limit groups are re}
The set $\overline{\cF}$ is recursively enumerable.
\end{theorem}

With these facts in hand, the algorithmic properties of $\overline{\cF}$ follow easily.

\begin{corollary}\label{cor: Factor sets are computable over F}
Factor sets are computable over $\cF$.
\end{corollary}
\begin{proof}
This follows immediately from Theorems \ref{thm: Decidable ut and re limit groups implies computable factor sets}, \ref{thm: Makanin's theorem} and \ref{thm: Limit groups are re}.
\end{proof}

\begin{corollary}[\cite{GrovesWilton09}]\label{cor: Limit groups are rmwp}
The set $\overline{\cF}$ is recursive modulo the word problem.
\end{corollary}
\begin{proof}
This follows immediately from Corollary \ref{cor: Limit group criterion for rmwp} and Theorems \ref{thm: Makanin's theorem} and \ref{thm: Limit groups are re}.
\end{proof}

To use this information to restrict to subclasses of $\overline{\cF}$, we use the fact that the isomorphism problem is solvable for $\cF$-limit groups \cite{BKM07,dahmanigroves1}.

\begin{theorem}
The isomorphism problem is solvable in $\overline{\cF}$.
\end{theorem}

In particular, it now follows easily that the class of free groups is recursive modulo the word problem, exactly as in \cite{GrovesWilton09}.

\begin{theorem}\label{thm: F is rmwp}
The class of groups $\cF$ is recursive modulo the word problem.
\end{theorem}
\begin{proof}
By Corollary \ref{cor: Limit groups are rmwp} and Lemma \ref{lem: recursive subsets 2}, it suffices to show that $\cF$ is recursive in $\overline{\cF}$.  But the isomorphism class of each element of $\cF$ is uniquely determined by its abelianization, so each $\cF$-limit group is isomorphic to at most one choice of free group, and this choice of free group is computable.  Using the solution to the isomorphism problem in $\overline{\cF}$, it follows that $\cF$ is recursive in $\overline{\cF}$ as required.
\end{proof}

As most surface groups are $\cF$-limit groups, a similar argument will show that $\cM_2$ is recursive modulo the word problem.  However, to deal with the few exceptions, we shall postpone the proof of this until we have considered virtual manifold groups.

\section{Virtual manifold groups}\label{sec:wall3d}

The goal of this section is to prove the following theorem,
conjectured by Wall to hold in all dimensions
\cite[p.\ 281]{farrell_borel_2002}. 

\begin{theorem}\label{t: Virtual}
If $\Gamma$ is a torsion-free group with a subgroup of finite index isomorphic to $\pi_1M$, where $M$ is a closed, irreducible 3--manifold, then $\Gamma$ is also the fundamental group of a closed, irreducible 3--manifold.
\end{theorem}

The bulk of the work was done by Zimmermann, who dealt with the case of an effective extension of a Haken manifold in \cite{Zi82}.  To deal with the remaining cases, we use Geometrization, Mostow Rigidity and the Convergence Group Theorem.

We will also give a proof of the corresponding (well known) result in two dimensions, which relies on the Nielsen Realization Problem.  

The utility of this theorem in our context lies in the following consequence.  For a class $\cC$ of finitely presentable groups, recall that $\mc{VC}$ is the class of \emph{virtually $\cC$ groups}---that is, the class of finitely presentable groups with a subgroup of finite index in $\mc{C}$.

\begin{corollary}\label{cor: AM3 in VAM3}
The class $\mc{AM}_3$ is recursive in $\mc{VAM}_3$.
\end{corollary}

To prove the corollary, we use the following observation.

\begin{lemma}\label{FTRE}
Let $\mc{FT}$ be the class of finitely presentable groups with a
non-trivial cyclic subgroup that injects into some finite quotient.
Then $\mc{FT}$ is recursively enumerable and contains every residually
finite group with torsion.  On the other hand, every torsion-free
group is contained in $\cU\smallsetminus\mc{FT}$.
\end{lemma}
\begin{proof}
  Note first that presentations of finite groups are recursively
  enumerable.  Since the word problem is uniformly solvable in finite
  groups, one can enumerate the pairs 
  $(\langle S\mid R\rangle ,w)$ where $\langle S\mid R\rangle$ presents a finite group, $R$ contains the word $w^n$ for
  some $n\geq 2$, and $w$ has order exactly $n$ in the group presented by
  $\langle S \mid R\rangle$.  Any element of $\mc{FT}$ is obtained
  from such a pair by deleting some of the elements of $R$ which are
  not equal to $w^n$.  Thus $\mc{FT}$ is recursively
  enumerable.
  
  Let $\Gamma$ be a residually finite group and $C=\langle w\rangle$ a
  non-trivial finite cyclic subgroup.  Replacing $C$ by a subgroup if
  necessary, we may assume $w$ has
  prime order.  If $\eta:\Gamma\to Q$ is a homomorphism to a finite group with $\eta(w)\neq 1$ then $\eta|C$ is an injection.  Therefore, $\Gamma\in \mc{FT}$.
  
  The final assertion is trivial.
\end{proof}

We also use the following consequence of the triangulability of
$3$--manifolds.
\begin{lemma}\label{m3re}
  The set of presentations of closed $3$--manifold groups $\cM_3$ is recursively enumerable.
\end{lemma}
\begin{proof}
  The collection of finite simplicial $3$--complexes (up to
  isomorphism of complex) is clearly
  recursively enumerable.  
  Moise showed that every $3$--manifold is
  triangulable \cite{Moise}, so every closed $3$--manifold appears in
  this collection.  Moreover, by checking the Euler characteristic and
  number of components of links of vertices, it is straightforward to
  check whether a given $3$--complex is a $3$--manifold.  Thus the set
  of triangulated $3$--manifolds is recursively enumerable.  Any
  maximal tree in the one-skeleton of a triangulated $3$--manifold
  gives a presentation if its fundamental group, so we can generate a
  list which contains at least one presentation of each $3$--manifold
  group.  By applying Tietze transformations to these presentations, we
  can enumerate all presentations of closed $3$--manifold groups.
\end{proof}

Finally, we will need to know that virtual $3$--manifold groups 
are residually finite.
\begin{proposition}\label{hempelcor}
  Every group whose presentations lie in $\mc{VM}_3$ is residually
  finite.  In particular, the word problem is uniformly solvable in
  $\mc{VM}_3$. 
\end{proposition}
\begin{proof}
  Hempel \cite{hempel87} proved that if $M$ is a $3$--manifold
  satisfying the Geometrization Conjecture, then $\pi_1M$ is
  residually finite.  Thus Perelman's work implies that every group in
  $\mc{M}_3$ is residually finite.  But a virtually residually finite
  group is residually finite.  
\end{proof}

\begin{proof}[Proof of Corollary \ref{cor: AM3 in VAM3}]
By Lemma \ref{m3re}, $\mc{M}_3$ is recursively enumerable.
Note that $\mc{AM}_3 = \cM_3\cap\mc{VAM}_3$, so we may set
$\cX_{\mc{AM}_3,\mc{VAM}_3}=\cM_3$. On the other hand, by Theorem
\ref{t: Virtual}, the complement $\mc{VAM}_3\smallsetminus\mc{AM}_3$
consists precisely of those groups in $\mc{VAM}_3$ with torsion.
Because every element of $\mc{VAM}_3$ is residually finite, we may set
$\cY_{\mc{AM}_3,\mc{VAM}_3}= \mc{FT}$, which is recursively enumerable
by Lemma \ref{FTRE}.
\end{proof}

Thus, in order to prove that (some recursive subset of) $\mc{AM}_3$ is recursive modulo the word problem, it often suffices to prove that (some recursive subset of) $\mc{VAM}_3$ is recursive modulo the word problem. 

\subsection{The Haken, non-Seifert-fibred case}

\begin{definition}
An extension $1\to K \to G \to Q\to 1$ is \emph{effective} if the induced map $Q\to \Out(K)$ is injective.
\end{definition}

We shall denote the centralizer of $K$ in $G$ by $Z_G(K)$, and the centre of $G$ by $Z(G)$.

\begin{lemma}\label{l:effective0}
  Let \[ 1\to K \longrightarrow G \longrightarrow Q \to 1\] be an extension.  If $Z_G(K)\subseteq K$ then the extension is effective.
\end{lemma}
\begin{proof}
Suppose $g\in G$ acts on $K$ as conjugation by $k_0$.  Then $k_0^{-1}g\in Z_G(K)$, so $g\in K$.  This shows that $Q=G/K$ injects into $\Out(K)$, as required.
\end{proof}

\begin{lemma}\label{l:effective}
 Let \[ 1\to K \longrightarrow G \longrightarrow Q \to 1\] be an extension.  If
 \begin{enumerate}
	\item $Z(K)=1$,  \label{centerless}
 	\item $Q$ is finite and \label{finite}
	\item $G$ is torsion-free \label{tf}
 \end{enumerate}
 then the extension is effective.
\end{lemma}
\begin{proof}
Let $\pi:G\to Q$ be the quotient map. By (\ref{centerless}), $\pi$ is injective on $Z_G(K)$.  Hence $Z_G(K)$ is finite by (\ref{finite}) and so, by (\ref{tf}), trivial.  It now follows from Lemma \ref{l:effective0} that the extension is effective.
\end{proof}

Here is a theorem of Zimmermann.
\begin{theorem}\label{t:zimmermann}
  \cite[Satz 0.2]{Zi82}  Let $M$ be a closed orientable Haken
  $3$--manifold, and $E$ a torsion-free effective extension
\[ 1\to \pi_1 M \to E \to F \to 1\]
  with $F$ finite.
  Then
  $E\cong \pi_1 M^*$, where $M^*$ is some closed irreducible
  $3$--manifold.\footnote{A further sentence reads:  ``If $M^*$ is not Haken, then it is either Seifert
  fibred or hyperbolic.'', but we don't need this.}
\end{theorem}

\begin{corollary}\label{c: Haken non-sf}
  Suppose $G$ is a torsion-free group and $G_0<G$ is a finite-index
  subgroup of $G$ isomorphic to $\pi_1 M$ for $M$ closed, orientable,
  Haken, and not
  Seifert fibred.  Then $G$ is the fundamental group of a closed
  $3$--manifold.
\end{corollary}
\begin{proof}
  By passing to a further finite-index subgroup, we may suppose $G_0$
  is normal.  There is a short exact sequence
\[ 1 \to G_0 \to G \to F \to 1 \]
  where $F = G/G_0$ is finite.    Since $M$ is not Seifert fibred,
  $Z(G_0)$ is trivial.  Lemma \ref{l:effective} implies that
  the extension is effective.  Theorem \ref{t:zimmermann} then gives
  the corollary.
\end{proof}

\subsection{The hyperbolic case}

To deal with the hyperbolic case, we will use a well known consequence of Mostow Rigidity.

\begin{theorem}[Mostow Rigidity]
Let $M$ be a hyperbolic $n$-manifold of finite volume, for $n\geq 3$.
Then every outer automorphism of $\pi_1 M$ is realized by an isometry of $M$.
\end{theorem}

\begin{corollary}\label{c: Hyperbolic case}
  Suppose $G$ is a torsion-free group and $G_0<G$ is a finite-index
  subgroup of $G$ isomorphic to $\pi_1 M$ for $M$ a closed, orientable,
  hyperbolic $m$--manifold.  Then $G$ is the fundamental group of a closed
  $m$--manifold.  
\end{corollary}
\begin{proof}
  Again, we may suppose that $G_0$ is normal and we have an extension:
\begin{equation}\label{g0gf} 1\to G_0 \to G \to F \to 1. \end{equation}
  Lemma \ref{l:effective} implies that the extension is effective.
  Mostow Rigidity implies that the outer action of $F$ on $G_0$ is
  induced by an isometric action of $F$ on the hyperbolic manifold
  $M$.  The quotient by this action is an orbifold $O$, whose
  (orbifold) fundamental group fits into an exact sequence
\begin{equation}\label{g0of} 1 \to G_0 \to \pi_1 O \to
  F \to 1. \end{equation}
  The outer actions of $F$ on $G_0$ are the same in \eqref{g0gf}
  and \eqref{g0of}.  Since $Z(G_0)$ is trivial, this implies that the
  extensions are equivalent and so $G \cong \pi_1 O$ (see, for instance, \cite[IV.6.8]{brown:cohomology}) .  If $\pi_1 O$ is
  torsion-free then the singular set of $O$ is empty, and so $O$ is a manifold.
\end{proof}

\subsection{Seifert fibred manifolds}

The following lemma will be useful.

\begin{lemma}\label{torsionfreeiffsfs}
  Suppose $G$ fits into an exact sequence
\[ 1\to Z \to G \to  \pi_1O\to 1 \]
  where $O$ is a closed 2--orbifold of non-positive Euler characteristic and $Z$ is infinite cyclic.  Then $G$ is the fundamental group of a Seifert fibred space if and only if $G$ is torsion-free.
\end{lemma}

The hyperbolic case is Lemma \ref{torsionfreeiffsfs hyperbolic} below, and the Euclidean case is Lemma \ref{torsionfreeiffsfs Euclidean} below.  The proof in the Euclidean case is similar to the proof in the hyperbolic case, but slightly more delicate.

\subsubsection{Hyperbolic base}

For the case of $M$ a Seifert fibred manifold with hyperbolic base, we apply the Convergence Group Theorem of Tukia, Gabai and Casson--Jungreis \cite{tukia_homeomorphic_1988,casson_convergence_1994,gabai_convergence_1992}.

\begin{theorem}\label{convergence}
  Suppose that a group $Q$ is quasi-isometric to $\bH^2$.  Then there is a short exact sequence
\[ 1\to F \to Q \to \Gamma \to 1\]
  where $F$ is finite and $\Gamma$ is a cocompact Fuchsian group.
\end{theorem}

The next lemma is the hyperbolic case of Lemma \ref{torsionfreeiffsfs}.

\begin{lemma}\label{torsionfreeiffsfs hyperbolic}
  Suppose $G$ fits into an exact sequence
\[ 1\to Z \to G \to  \pi_1O\to 1 \]
  where $O$ is a closed, hyperbolic 2--orbifold and $Z$ is infinite cyclic.  Then $G$ is the fundamental group of a  Seifert fibred space if and only if $G$ is torsion-free.
\end{lemma}
\begin{proof}
If $G$ is the fundamental group of a Seifert fibred space $M$ with
hyperbolic base, then $M$ is an aspherical manifold so $G$ is torsion-free.  We will now prove the converse.

Let $\pi:G\to\pi_1O$ be the quotient homomorphism and let $\omega:\pi_1O\to\Aut(Z)\cong\bZ/2$ be the induced action of $\pi_1O$ on $Z$.  By Selberg's Lemma, there is a torsion-free normal subgroup $\pi_1\Sigma$ of $\pi_1O$ contained in $\ker\omega$, where $\Sigma$ is a closed, orientable, hyperbolic surface that covers $O$.  Let $K=\pi^{-1}(\pi_1\Sigma)$, which fits into the short exact sequence
\[
1\to Z\to K\to\pi_1\Sigma\to 1.
\]
Let $F\cong G/K\cong\pi_1O/\pi_1\Sigma$.  Then $K$ is isomorphic to the fundamental group of a circle bundle over $\Sigma$, which we shall denote by $M$.  Because $\Sigma$ is orientable and $\pi_1\Sigma$ acts trivially on $Z$, $M$ is orientable.  Furthermore, $M$ is Haken because $\Sigma$ is hyperbolic.

If $g\in Z_G(K)$ then $\pi(g)\in Z_{\pi_1O}(\pi_1\Sigma)=1$ so $g\in\ker\pi= Z$.   This shows that $Z_G(K)\subseteq Z\subseteq K$, and so, by Lemma \ref{l:effective0}, the extension
\[
1\to K\to G\to F\to 1
\]
is effective.  Therefore, it follows from Zimmermann's Theorem \ref{t:zimmermann} that $G$ is the fundamental group of a closed 3--manifold, which is necessarily Seifert fibred as $G$ has an infinite cyclic normal subgroup.
\end{proof}

The next corollary is the main result of this subsection.

\begin{corollary} \label{c: Hyperbolic base}
  Suppose $G$ is a torsion-free group and $G_0<G$ is a finite-index
  subgroup of $G$ isomorphic to $\pi_1 M$ for $M$ Seifert fibred with
  hyperbolic base orbifold.  Then $G$ is the fundamental group of a
  closed $3$--manifold.
\end{corollary}
\begin{proof}
  Again we may assume that $G_0$ is normal.  Because $M$ is Seifert fibred with hyperbolic base, the fibre subgroup $Z_0$ is a characteristic infinite cyclic subgroup of $G_0$ so that $Q_0=G_0/Z_0$ is cocompact Fuchsian.  Since $Z_0$ is characteristic and $G_0$ is
  normal, $Z_0$ is also normal in $G$.  Because $G_0$ has finite index in $G$, $Q_0$ is a cocompact Fuchsian subgroup of finite index in   $Q = G/Z_0$.  Theorem  \ref{convergence} then implies that there is a short exact sequence:
\[ 1\to F \to Q \to \Gamma \to 1 \]
  where $\Gamma$ is cocompact Fuchsian.
  Let $\pi \co G \to Q = G/Z_0$ be the natural projection map.  Then
  $Z :=\pi^{-1}(F)$ is a virtually infinite cyclic, normal subgroup of $G$.  
  Since $G$ is torsion-free, $Z$ is infinite cyclic, and we have
\[ 1\to Z \to G \to \Gamma \to 1. \]
  Lemma \ref{torsionfreeiffsfs hyperbolic} now finishes the argument.
\end{proof}

\subsubsection{Euclidean base}

To deal with the case of a Euclidean base, we start with the following result, presumably known to Bieberbach, which serves as a replacement for the Convergence Subgroup Theorem in our arguments.

\begin{proposition} \label{p:Act on R^k}
Suppose that $G$ is a group with a finite-index subgroup isomorphic to $\bZ^k$ for some $k \ge 0$.  Then there is a short exact sequence
\[ 1\to F\to G\to \Gamma\to 1 \]
where $F$ is a finite group and $\Gamma$ acts properly discontinuously and cocompactly on $\bE^k$. 
\end{proposition}
\begin{proof}
We may assume that the subgroup $A$ isomorphic to $\bZ^k$ is normal, so we have a short exact sequence
 \[	1 \to A \to G \to Q \to 1		\]
where $Q$ is finite. Because wreath products contain all extensions \cite{krasner_produit_1951}, $G$ can be regarded as a subgroup of the wreath product $A \wr Q$.   By definition, the quotient group $Q$ permutes the factors of the base $B=\bigoplus_{q\in Q} A\cong \bZ^{k|Q|}$.  Therefore, $A \wr Q$ acts properly discontinuously and cocompactly on $\bE^m$, where $m=k|Q|$, and the base $B$ is precisely the subgroup of pure translations. In particular, $G$ acts properly discontinuously on $\bE^m$, and $A$ is the subgroup of pure translations.

Let $V$ be the unique minimal $A$-invariant vector subspace of $\bE^m$.  Clearly $V$ is the $\bR$-span of the $A$-orbit of the origin.  We claim that $V$ is in fact $G$-invariant.   Indeed, let $g\in G$ and $a\in A$.  Then $(ga).0$ lies on the axis of $gag^{-1}$.  But $gag^{-1}$ is a pure translation, so $(ga).0\in V$.   It follows that $gV=V$, so $V$ is $G$-invariant as claimed.

We therefore have a properly discontinuous and cocompact action of $G$ on $V\cong \bE^k$.  The kernel $F$ of this action intersects $A$ trivially, and so maps injectively into $Q$.  Therefore $F$ is finite.
\end{proof}

The Euclidean case of the main result follows quickly in all dimensions.

\begin{corollary}\label{c: Euclidean case}
Suppose $G$ is a torsion-free group and $G_0<G$ is a finite-index subgroup of $G$ isomorphic to $\pi_1 M$ for $M$ a closed, Euclidean $m$--manifold.  Then $G$ is the fundamental group of a closed, Euclidean $m$--manifold.
\end{corollary}
\begin{proof}
By Bieberbach's Theorem, $G_0$ has a subgroup of finite index which is isomorphic to $\bZ^m$.  By Lemma \ref{p:Act on R^k}, $G$ acts properly discontinuously and cocompactly on $\bE^k$ with finite kernel; but $G$ is torsion-free, so the kernel of the action is in fact trivial.  Now $O=\bE^m/G$ is an $m$--dimensional orbifold.  Because $G$ is torsion-free, the singular set of $O$ is empty, so $O$ is a manifold.
\end{proof}

For the case of Nil-geometry, we will need the Euclidean analogue of Lemma \ref{torsionfreeiffsfs hyperbolic}.  First, however, we prove a useful fact about centralizers.

\begin{lemma}\label{l: Euclidean centralizers}
Let $G$ be a group that acts properly discontinuously and cocompactly on $\bE^k$, let $A$ be the finite-index subgroup of pure translations and let $B$ be any subgroup of finite index in $A$.  Then $Z_G(B)=A$.
\end{lemma}
\begin{proof}
Because $A$ is abelian and contains $B$ it is clear that $A\subseteq Z_G(B)$.  For the reverse inclusion, consider $g\in G\smallsetminus A$.  Suppose that $r(g)$ is the rotational part of $g$.  Then $r(g)$ is of finite order, fixing some point $x_0\in\bE^k$.  Suppose that $g \in Z_G(B)$.  We wish to prove that $r(g) = 1$.  It is easy to see that $g$ centralizes $B$ if and only if $r(g)$ does.    Let $S(x_0)$ be the unit tangent sphere at $x_0$.  For any $b\in B\smallsetminus 1$, let $l_b$ be the unique axis of $b$ that passes through $x_0$.  Because $r(g).l_b=l_{r(g)br(g)^{-1}}$, if $r(g)$ and $b$ commute then $r(g).l_b=l_b$, and furthermore $r(g)$ fixes $l_b$ pointwise.  But the set of points
\[
\bigcup_{b\in B\smallsetminus 1} l_b\cap S(x_0)
\] 
is dense in $S(x_0)$ because $B$ acts cocompactly on $\bE^k$.  Therefore $r(g)$ acts trivially on $S(x_0)$, and so $r(g)=1$ and $g \in A$, as required.
\end{proof}

We are  now ready to prove the Euclidean case of Lemma \ref{torsionfreeiffsfs}.  The proof is similar to the hyperbolic case, the main difference being that $\Sigma$ needs to be chosen more carefully.

\begin{lemma}\label{torsionfreeiffsfs Euclidean}
  Suppose $G$ fits into an exact sequence
\[ 1\to Z \to G \to  \pi_1O\to 1 \]
  where $O$ is a closed Euclidean 2--orbifold and $Z$ is infinite cyclic.  Then $G$ is the
  fundamental group of an orientable Seifert fibred space if and only if
  $G$ is torsion-free.
\end{lemma}
\begin{proof}
Just as in the proof of Lemma \ref{torsionfreeiffsfs hyperbolic}, if $G$ is the fundamental group of a Seifert fibred space then it is torsion-free.  We will now prove the converse.

Let $\pi:G\to\pi_1O$ be the quotient homomorphism and let $\omega:\pi_1O\to\Aut(Z)\cong\bZ/2$ be the induced action of $\pi_1O$ on $Z$.  Let $A$ be the normal subgroup of finite index in $\pi_1O$ consisting of pure translations and let $\pi_1\Sigma=\ker\omega\cap A$.  Here $\Sigma$ is a closed, orientable surface that covers $O$.  Let $K=\pi^{-1}(\pi_1\Sigma)$, which fits into the short exact sequence
\[
1\to Z\to K\to\pi_1\Sigma\to 1.
\]
Let $F\cong G/K\cong\pi_1O/\pi_1\Sigma$.  Then $K$ is isomorphic to the fundamental group of a circle bundle over $\Sigma$, which we shall denote by $M$.  Because $\Sigma$ is orientable and $\pi_1\Sigma$ acts trivially on $Z$, $M$ is orientable.  Furthermore, $M$ is irreducible because $\Sigma$ is Euclidean.

If $g\in Z_G(K)$ then we have that $\pi(g)\in Z_{\pi_1O}(\pi_1\Sigma)=A$, where the last equality is given by Lemma \ref{l: Euclidean centralizers}.  But it is also true that $\pi(g)\in\ker\omega$, because $g\in Z_G(K)\subseteq Z_G(Z)$.  Therefore, $\pi(g)\in \pi_1\Sigma$.  This shows that $Z_G(K)\subseteq K$, and so, by Lemma \ref{l:effective0}, the extension
\[
1\to K\to G\to F\to 1
\]
is effective.  It now follows from Zimmermann's Theorem that $G$ is the fundamental group of a closed 3--manifold, which is necessarily Seifert fibred as $G$ has an infinite cyclic normal subgroup.
\end{proof}

The Nil-geometry case of the main result now follows as in the case of a hyperbolic base.

\begin{corollary} \label{c: Nil case}
  Suppose $G$ is a torsion-free group and $G_0<G$ is a finite-index
  subgroup of $G$ isomorphic to $\pi_1 M$ for $M$ a 3--manifold with
  Nil geometry.  Then $G$ is the fundamental group of a
  closed $3$--manifold.
\end{corollary}
\begin{proof}
  Again we may assume that $G_0$ is normal.  The fibre subgroup $Z_0$ is a characteristic infinite cyclic subgroup of $G_0$ such that $Q_0=G_0/Z_0$ is the fundamental group of a closed, Euclidean 2--orbifold.  Since $Z_0$ is characteristic and $G_0$ is  normal, $Z_0$ is also normal in $G$.  Because $G_0$ has finite index in $G$, $Q_0$ is the fundamental group of a closed, Euclidean 2--orbifold $O$ and has finite index in   $Q = G/Z_0$.  Proposition \ref{p:Act on R^k} then implies that there is a short exact sequence:
\[ 1\to F \to Q \to \Gamma \to 1. \]
  Let $\pi \co G \to Q = G/Z_0$ be the natural projection map.  Then
  $Z :=\pi^{-1}(F)$ is a virtually infinite cyclic, normal subgroup of $G$.  
  Since $G$ is torsion-free, $Z$ is infinite cyclic, and we have
\[ 1\to Z \to G \to \Gamma \to 1 \]
  where $\Gamma=\pi_1O$.  Lemma \ref{torsionfreeiffsfs Euclidean} now finishes the argument.
\end{proof}

\subsection{The general 3-dimensional case}

The following is a consequence of the Geometrization Theorem.

\begin{proposition}\label{5chotomy}
  Let $M$ be an irreducible orientable $3$--manifold with infinite
  fundamental group.  Then at least one of the
  following holds:
  \begin{enumerate}
  \item\label{hak} $M$ is Haken, and $Z(\pi_1 M)$ is trivial.
  \item\label{hyp} $M$ is hyperbolic.
  \item\label{sfsh} $M$ is Seifert fibred with hyperbolic base orbifold.
  \item\label{eucl} $M$ admits Euclidean geometry.
  \item\label{nil} $M$ admits Nil geometry.
  \end{enumerate}
\end{proposition}

Thus the result we want follows from the results to be proved in the previous four
subsections.  

\begin{proof}[Proof of Theorem \ref{t: Virtual}]
There are five cases to consider, by Proposition \ref{5chotomy}.  The case of Haken $M$ with $Z(\pi_1M)=1$ follows from Corollary \ref{c: Haken non-sf}.    The hyperbolic case follows from Corollary \ref{c: Hyperbolic case}.  The Seifert-fibred-with-hyperbolic-base-orbifold case is dealt with by Corollary \ref{c: Hyperbolic base}.  The Euclidean case is handled by Corollary \ref{c: Euclidean case}. The case of Nil geometry follows from Corollary \ref{c: Nil case}.
\end{proof}

\subsection{The 2-dimensional case}

We also record here the 2-dimensional version of the theorem.

\begin{theorem}\label{thm: 2d virtual}
If $\Gamma$ is a torsion-free group with a subgroup of finite index isomorphic to $\pi_1M$, where $M$ is a closed, aspherical 2--manifold, then $\Gamma$ is also the fundamental group of a closed, aspherical 2--manifold.
\end{theorem}
\begin{proof}
The case of Euclidean $M$ is dealt with by Corollary \ref{c: Euclidean case}, so we may assume that $M$ is hyperbolic.  The proof is now exactly the same as the proof of Corollary \ref{c: Hyperbolic case}, with the Nielsen Realization Theorem used instead of Mostow Rigidity.
\end{proof}

\section{Surface groups}\label{sec: Surface groups}

In this section, we prove that the classes of surface groups and aspherical surface groups are recursive modulo the word problem.  We will use the fact, already noted in Section \ref{s:limit groups}, that free groups and the fundamental groups of closed, orientable surfaces are $\cF$-limit groups.

That the class of orientable surface groups is recursive modulo the word problem follows easily.

\begin{lemma}\label{lem: Orientable surface groups}
The classes $\mc{AM}_2^+$ and $\cM_2^+$ are recursive modulo the word problem.
\end{lemma}
\begin{proof}
By Corollary \ref{cor: Limit groups are rmwp}, it suffices to prove that $\cM_2^+$ is recursive in $\overline{\cF}$.  Groups in $\cM_2^+$ are determined by their abelianizations, so any limit group $G$ is isomorphic to at most one possible orientable surface group $\Gamma$, and one can compute a presentation for $\Gamma$ from a presentation for $G$.  By the solution to the isomorphism problem for limit groups \cite{BKM07, dahmanigroves1}, there is an algorithm to determine whether or not $\Gamma\cong G$.  This completes the proof that $\cM_2^+$ is recursive modulo the word problem.  Because $\cM_2^+=\{1\}\cup\mc{AM}_2^+$, the result also follows for $\mc{AM}_2^+$.
\end{proof}

Using Theorem \ref{thm: 2d virtual}, we can now improve this result to all surface groups.

\begin{theorem}\label{thm: Surfaces rmwp}
The classes $\cM_2$ and $\mc{AM}_2$ are recursive modulo the word problem.
\end{theorem}
\begin{proof}
Because $\cM_2=\{1,\bZ/2\}\cup\mc{AM}_2$ and the former is a finite class of abelian groups, it suffices to prove the theorem for $\mc{AM}_2$.

By Lemma \ref{lem: Orientable surface groups}, $\mc{AM}_2^+$ is recursive modulo the word problem, and so $\cV_2\mc{AM}_2^+$ is also by Lemma \ref{lem: Virtually recursive modulo the word problem}.  Evidently $\mc{AM}_2\subseteq\cV_2\mc{AM}_2^+$.  The theorem now follows from the claim that $\mc{AM}_2$ is recursive in $\cV_2\mc{AM}_2^+$.

By the classification of surfaces, $\mc{AM}_2$ is recursively enumerable, so we can set
\[
\cX_{\mc{AM}_2,\cV_2\mc{AM}_2^+}=\mc{AM}_2~.
\]
On the other hand, by Theorem \ref{thm: 2d virtual}, every group in
$\cV_2\mc{AM}_2^+\smallsetminus \mc{AM}_2$ has torsion.  Because the
groups in $\cV_2\mc{AM}_2^+$ are residually finite, it follows that we
can take $\cY_{\mc{AM}_2,\cV_2\mc{AM}_2^+}$ to be the class of groups
with a homomorphism that embeds a non-trivial cyclic subgroup into a
finite quotient.  This class of groups is recursively enumerable by
Lemma \ref{FTRE}, and only contains groups with torsion, so $\cY_{\mc{AM}_2,\cV_2\mc{AM}_2^+}\cap\mc{AM}_2=\varnothing$.  This proves that $\mc{AM}_2$ is recursive in $\cV_2\mc{AM}_2^+$, and the theorem follows.
\end{proof}

\section{Spherical geometry}\label{sec:spherical}

In this section, we will prove that $\cSF_{+,\neq}$, the class of fundamental groups of 3--manifolds with spherical geometry, is recursive modulo the word problem.  The fundamental tool is the following lemma.

\begin{lemma}\label{l:vmeta}
  If $\Gamma\in\cSF_{+,\neq}$ then $\Gamma$ has a cyclic subgroup of index at most $240$.
\end{lemma}
\begin{proof}
  Let $\Gamma = \pi_1 M$ where $M$ is a spherical $3$--manifold, which therefore is Seifert fibred with positive-Euler-characteristic base.  By passing to a double cover, we can assume the base is orientable, hence a sphere with at most $3$ cone points.  If there are $2$ or fewer, this cover is a lens space and has cyclic fundamental group.  If there are $3$, then the base orbifold is $S^2$ modulo a subgroup of a tetrahedral, octahedral, or icosahedral group.  Thus the fundamental group of the base has order at most $120$.  Passing to a cover where the base orbifold is $S^2$, we get a lens space again.  Since $M$ has a lens space cover of degree at most $240$, $\Gamma$ has a cyclic subgroup of index at most $240$.
\end{proof}

Let $\mc{FC}$ be the class of finite cyclic groups. Note that $\mc{FC}$ is a recursive subset of $\cA$, and hence is recursive modulo the word problem.

\begin{theorem}\label{thm: SF+neq}
The class $\cSF_{+,\neq}$ is recursive modulo the word problem.
\end{theorem}
\begin{proof}
We have that
\[
\cSF_{+,\neq}\subseteq \cV_{240}\mc{FC}
\]
and the latter class is recursive modulo the word problem. It remains only to prove that $\cSF_{+,\neq}$ is recursive in $\cV_{240}\mc{FC}$.  Since any 3--manifold with finite fundamental group is in $\cSF_{+,\neq}$, we can take
\[
\cX_{\cSF_{+,\neq},\cV_{240}\mc{FC}} =\cM_3~.
\]
It remains to define $\cY_{\cSF_{+,\neq},\cV_{240}\mc{FC}}$.  The point here is that, for each $n$, the sets
\[
\cV_{240}\{\bZ/n\}
\]
are uniformly recursive in $\cV_{240}\mc{FC}$.  On the other hand, for each $n$, the intersection $\cV_{240}\{\bZ/n\}\cap\cM_3$ is a computable set of finitely many isomorphism classes \cite[Theorem 4.11]{scott:geom}.  Since the isomorphism problem is solvable in the class of finite groups, the set
\[
\cY_{\cSF_{+,\neq},\cV_{240}\mc{FC}}=\bigcup_{n\in\bN} (\cV_{240}\{\bZ/n\}\smallsetminus\cM_3)=\cV_{240}\mc{FC}\smallsetminus\cM_3
\]
is recursively enumerable.
\end{proof}

\section{$S^2\times\bR$ geometry}\label{sec:S2XR}

This case is very simple.

\begin{lemma}\label{lem: Classification of SF+=}
If $\Gamma=\pi_1M$ and $M$ is a closed 3--manifold with $S^2\times\bR$ geometry then
\[
\Gamma\in \{\bZ,\bZ\times\bZ/2,\bZ/2\ast\bZ/2\}
\]
and each of these possibilities is realized.
\end{lemma}
\begin{proof}
See, for instance, \cite{scott:geom}.
\end{proof}

\begin{lemma}
We have
\[
\cSF_{+,=}=\cV_2\{\bZ\}~.
\]
\end{lemma}
\begin{proof}
Any $\Gamma\in\cV_2\{\bZ\}$ fits into a short exact sequence
\[
1\to\bZ\to\Gamma\to\bZ/2\to 1~.
\]
It is easy to see that any finite cyclic subgroup of $\Gamma$ embeds into $\bZ/2$.  Therefore, either $\Gamma$ is torsion-free, hence $\bZ$, or the sequence splits and $\Gamma$ is isomorphic to either $\bZ\times\bZ/2$ or the infinite dihedral group.  The result now follows from Lemma \ref{lem: Classification of SF+=}.
\end{proof}

\begin{theorem}\label{thm: SF+=}
The class $\cSF_{+,=}$ is recursive modulo the word problem.
\end{theorem}
\begin{proof}
The class $\{\bZ\}$ is a recursive subset of $\cA$ and so is recursive modulo the word problem.  Therefore $\cSF_{+,=}=\cV_2\{\bZ\}$ is also recursive modulo the word problem.
\end{proof}

\section{Euclidean geometry}\label{sec:euclid}

The following lemma will be useful.

\begin{lemma}
If $O$ is a closed, Euclidean $2$--orbifold without corner reflectors then $O$ has a cover of index at most $6$ which is homeomorphic to a $2$--torus. 
\end{lemma}
\begin{proof}
Any such $O$ is on the list in Table 4.4 of \cite{scott:geom}.  By
inspection, $O$ has a covering space of degree at most $6$ which is
homeomorphic to a $2$--torus
\end{proof}

\begin{lemma}\label{lem:boundedindex3torus}
If $M$ is a closed, Euclidean 3--manifold then $M$ has a cover of index at most $12$ which is homeomorphic to a 3--torus. 
\end{lemma}
\begin{proof}
As explained in \cite{scott:geom}, $\pi_1M$ fits into a short exact sequence
\[
1\to\bZ\to\pi_1M\to\pi_1O\to 1
\]
where $O$ is a closed $2$--dimensional Euclidean orbifold without corner reflectors.  By the previous lemma, $O$ has a covering space $\Sigma$ of degree at most $6$ which is homeomorphic to a $2$--torus.  Passing to a further subgroup of index two, we may assume that the action of $\pi_1\Sigma$ on $\bZ$ is trivial.  The corresponding covering space $M'$ of $M$ is therefore a circle bundle over a torus with trivial monodromy, which covers $M$ with degree at most $12$.  Because $M'$ is Euclidean, this bundle is actually a trivial bundle, so $M'$ is homeomorphic to the $3$--torus.
\end{proof}

\begin{theorem}\label{thm: SF0=}
The class $\cSF_{0,=}$ is recursive modulo the word problem.
\end{theorem}
\begin{proof}
By Lemma \ref{lem:boundedindex3torus}, $\cSF_{0,=}\subseteq\cV_{24}\{\bZ^3\}$.  The set
$\{\bZ^3\}$ is recursive modulo the word problem by Corollary
\ref{abeliansingleton}.  Applying Lemma
\ref{lem: Virtually recursive modulo the word problem} with $f$ equal
to the constant function $12$, $\cV_{12}\{\bZ^3\}$ is recursive modulo
the word problem.  Since the word problem is uniformly solvable in
$\mc{VM}_3$ (Proposition \ref{hempelcor}),
and $\cV_{12}\{\bZ^3\}\subseteq\mc{VAM}_3$,
the set $\cV_{12}\{\bZ^3\}$ is recursive in $\mc{VAM}_3$.
Moreover $\cSF_{0,=}=\mc{AM}_3\cap\cV_{12}\{\bZ^3\}$.  By Corollary
\ref{cor: AM3 in VAM3}, $\mc{AM}_3$ is recursive in $\mc{VAM}_3$ and
so the result follows. 
\end{proof}

\section{$\mathrm{Nil}$ geometry}\label{sec:Nil}

As in the Euclidean case, we can pass to a cover of bounded degree which is a genuine circle bundle.

\begin{lemma}\label{lem:nilvirtualbundle}
If $M$ is a closed 3--manifold with $\mathrm{Nil}$ geometry then $M$ has a cover of index at most $12$ which is homeomorphic to an orientable circle bundle over a 2--torus. 
\end{lemma}

Let $\mc{B}^+_{0}$ be the class of fundamental groups of orientable circle bundles over the 2--torus, and let $\mc{B}^+_{0,\neq}$ be the subclass consisting of the fundamental groups of those bundles with $\mathrm{Nil}$ geometry.   The elements of $\mc{B}^+_{0}$ are the groups with presentations
\[
\Gamma_e=\langle a,b,z\mid [a,b]z^{-e},~[a,z],~[b,z]\rangle
\]
where $e\in\bN$ is the Euler class of the bundle.  Because
\[
H_1(\Gamma_e)\cong\bZ^2\times\bZ/e
\]
the Euler class $e$ can be computed from the abelianization.

The main point is to prove the following.

\begin{lemma}\label{lem: Nil bundles}
The classes $\mc{B}^+_{0,\neq}$ and $\mc{B}^+_{0}$ are recursive modulo the word problem.
\end{lemma}
\begin{proof}
First, note that $\cB^+_0\subseteq\cN_2$, the set of $2$--nilpotent groups.  Furthermore, $\cB^+_0=\{\bZ^3\}\sqcup\cB^+_{0,\neq}$.  Segal proved that the isomorphism problem is decidable in virtually polycyclic groups, in particular in $\cN_k$ for any $k$.  Therefore, $\{\bZ^3\}$ is recursive in $\cN_2$, and hence it suffices to prove that $\cB^+_0$ is recursive in $\cN_2$.  By the previous paragraph, a group $G$ with a presentation in $\cN_2$ can be isomorphic to at most one $\Gamma_e$, and that $e$ is computable from a presentation for $G$.  Because the isomorphism problem is solvable in $\cN_2$ by \cite{Segal90}, the result follows.
\end{proof}

The proof of the theorem is then similar to the Euclidean case.

\begin{theorem}\label{thm: SF0neq}
The class $\cSF_{0,\neq}$ is recursive modulo the word problem.
\end{theorem}
\begin{proof}
By Lemma \ref{lem:nilvirtualbundle} and Corollary \ref{cor: AM3 in
  VAM3}, we have that $\cSF_{0,\neq}= \cV_{12}\mc{B}_{0,\neq}\cap\mc{AM}_3$.  
By Lemmas \ref{lem: Nil bundles} and \ref{lem: Virtually recursive
  modulo the word problem}, the class $\cV_{12}\mc{B}_{0,\neq}$ is
recursive modulo the word problem.  In particular, it is recursive in
$\mc{VAM}_3$.
As $\mc{AM}_3$ is also recursive in
$\mc{VAM}_3$, so is the intersection $\cSF_{0,\neq}$.  It follows  (Lemma
\ref{lem: Recursive subsets of recursive subsets}.\eqref{basic2}) that the
set $\cSF_{0,\neq}$ is recursive in $\cV_{12}\mc{B}_{0,\neq}$.  As
$\cV_{12}\mc{B}_{0,\neq}$ is recursive modulo the word problem, so too
is $\cSF_{0,\neq}$, using Lemma \ref{lem: recursive subsets
  2}.
  \end{proof}

\section{$\mathrm{Sol}$ geometry}\label{sec:Sol}

The goal of this section is to prove that $\mc{SOL}$ is recursive modulo the word problem.  Let $\mc{TB}^2_{S^1}$ be the set of orientable 2-torus bundles over the circle.  As in the previous cases, we use an explicit upper bound on the degree of a torus-bundle cover.

\begin{lemma}
If $M$ is a closed 3--manifold with $\mathrm{Sol}$ geometry then $M$ has a cover of index at most $8$ which is homeomorphic to an orientable torus bundle over a circle.
\end{lemma}
\begin{proof}
The fundamental group of $M$ is a lattice in $\mathrm{Isom}(\mathrm{Sol})$, the isometry group of $\mathrm{Sol}$; $\mathrm{Sol}$ is naturally identified with the identity component of $\mathrm{Isom}(\mathrm{Sol})$, and has index 8 \cite[p. 476]{scott:geom}.  There is a short exact sequence
\[
1\to\bR^2\to\mathrm{Sol}\to\bR\to 1
\]
(see \cite[p. 470]{scott:geom}), and it follows that $\pi_1M\cap\mathrm{Sol}$ is the fundamental group of an orientable torus bundle.
\end{proof}

We use a structural theorem of Bieri--Strebel \cite[Theorem C]{BieriStrebel78}.

\begin{theorem}\label{thm:bieristrebel}
Every finitely presented abelian-by-cyclic group $G$ is an ascending HNN extension of a finitely generated abelian group $A$.  Furthermore, $G$ is polycyclic if and only if the endomorphism is an isomorphism.
\end{theorem}

\begin{remark}
The `only if' part of the above theorem is not contained in \cite{BieriStrebel78}, but it is proved in  \cite{BieriStrebel80}.
\end{remark}

Consider $\mc{ABC}$, the class of abelian-by-cyclic groups, $\cS_2$
the class of metabelian groups and  $\cF_{\cA}(\{\bZ\})$ the class of
groups with abelianization isomorphic to $\bZ$.  (Recall that in each
of these cases we really mean just those \emph{finite} presentations
giving a group with the given properties.)  Note that
$\cF_{\cA}(\{\bZ\})$ is recursive.  

\begin{lemma}
We have $\mc{ABC}\cap \cF_{\cA}(\{\bZ\})=\cS_2\cap \cF_{\cA}(\{\bZ\})$ and this class is recursive modulo the word problem.
\end{lemma}
\begin{proof}
Because $\mc{ABC}\subseteq \cS_2$, one direction of the equality is immediate.  On the other hand, any group in $\cS_2\cap \cF_{\cA}(\{\bZ\})$ is clearly abelian-by-cyclic.

The class $\cF_{\cA}(\{\bZ\})$ is recursive and $\mc{ABC}$ is
recursively enumerable by the Bieri-Strebel theorem
\ref{thm:bieristrebel}.  On the other hand, let $\cD$ be a class of
groups in which the word problem is uniformly solvable.  For any group
$G$, the metabelianization $G/G^{(2)}$ has a natural recursive
presentation.  Because the word problem is uniformly solvable for
metabelian groups (see eg \cite[Theorem 2.1]{BCR94}), a solution to the word problem in $G$ provides a partial algorithm that finds non-trivial elements of the kernel of the natural map $G\to G/G^{(2)}$.  Therefore, there is a partial algorithm that confirms membership of $\cD\smallsetminus \cS_2$.  The result follows.  
\end{proof}

\begin{theorem}\label{thm: SOL}
The class $\mc{SOL}$ is recursive modulo the word problem.
\end{theorem}
\begin{proof}
Because $\mc{SOL}\subseteq\cV_8(\mc{SOL}\cap\mc{TB}^2_{S^1})$, it
suffices to prove that $\mc{SOL}\cap\mc{TB}^2_{S^1}$ is recursive
modulo the word problem.   By the previous lemma, it suffices to show
that $\mc{SOL}\cap\mc{TB}^2_{S^1}$ is recursive in $\mc{ABC}\cap
\cF_{\cA}(\{\bZ\})$.  For a group $G\in \mc{ABC}\cap
\cF_{\cA}(\{\bZ\})$, the decomposition of $G$ as an ascending HNN
extension $A*_\phi$ is unique, and a non-deterministic search will
find it.  The class $\mc{SOL}\cap\mc{TB}^2_{S^1}$ consists of
precisely those groups for which $A\cong\bZ^2$ and $\phi$ is an
isomorphism of order greater than $6$ (implying infinite order).  As
both of these are decidable, the result follows. 
\end{proof}

\section{Seifert fibred manifolds with hyperbolic base orbifold}\label{sec:SF-}

Similarly to the results above, we approach $\cSF_{-,=}$ and
$\cSF_{-,\neq}$ via the related classes $\cB_{-,=}$ and $\cB_{-,\neq}$
of circle bundles over orientable surfaces with zero and non-zero
Euler numbers respectively.  Adapting the above techniques, we can use
algebraic geometry over free groups to prove that $\cB_{-,=}$  is
recursive modulo the word problem.  Although every Seifert fibred
manifold has a finite-sheeted covering space that is a circle bundle,
the degree of the covering map is not bounded independently of the manifold.  We will prove that the index of this subgroup is, however, computable, using an effective version of Selberg's Lemma.

\subsection{Circle bundles}

Let $\mc{CAS}$ be the class of finitely presentable central extensions of fundamental groups of closed, orientable, hyperbolic surfaces.

\begin{lemma}\label{l: Finite generation in central extensions}
Let
\[
1\to A\to G\to Q\to 1
\]
be a central extension.  If $G$ is finitely generated and $Q$ is finitely presented then $A$ is finitely generated.
\end{lemma}
\begin{proof}
Fix a generating set $X$ for $G$ and let $\langle X\mid R(X)\rangle$
be a finite presentation for $Q$.  Let $\langle Y\mid S(Y)\rangle$ be
a (not necessarily finite) presentation for $A$, and choose a set of
representative words $Y(X)$ to define a choice of map of free groups
$F_Y\to F_X$.  Thinking of $R(X)$ as a subset of the free group $F_X$,
we have by hypothesis that $R(X)\subseteq \llangle Y(X)\rrangle$.  Of
course $A = \langle Y(X)\rangle$ is normal, so 
$R(X) \subseteq \langle Y(X)\rangle$, and we may choose words $\widetilde{R}(Y)$ such that $\widetilde{R}(Y(X))=R(X)$.  The following is now easily checked to be a presentation for $G$.
\[
G\cong \langle X\mid S(Y(X))=1,~\widetilde{R}(Y(X))=R(X) \rangle
\]
Because $Q$ is finitely presented, only finitely many elements of $Y$ appear in the words $\widetilde{R}(Y)$.  Let $Y'\subseteq Y$ be the set of all words that appear.  Adjoining the relations $Y'(X)=1$ to $G$ defines the quotient map $G\to Q$, so the elements of $G$ defined by $Y'(X)$ normally generate $A$.  But $A$ is central in $G$, and so $Y'$ is a finite generating set for $A$.
\end{proof}

The following is an immediate consequence.

\begin{lemma}
The set $\mc{CAS}$ is recursively enumerable.
\end{lemma}

An easy argument characterizes factor sets (over free groups) of elements of $\mc{CAS}$.

\begin{lemma}
If $\Gamma\in\mc{CAS}$ then $F_{\cF}(\Gamma)\subseteq \{H_1(\Gamma;\bZ),\pi_1\Sigma\}$ for $\Sigma$ some closed, orientable, hyperbolic surface.
\end{lemma}

\begin{lemma}
The set $\mc{CAS}$ is recursive modulo the word problem.
\end{lemma}
\begin{proof}
By Corollary \ref{cor: Factor sets are computable over F}, the set of presentations of groups with $\cF$-factor sets consisting of the abelianization and a single map to a closed, orientable, hyperbolic surface group is recursive, so we may reduce to this case.  Thus, we only need to consider groups $G$ equipped with a given canonical epimorphism $G\to\pi_1\Sigma$.

Suppose now that the word problem is uniformly solvable in $\cD$. By the above lemma, we can set $\cX_{\mc{CA}(\pi_1\Sigma),\cD}=\mc{CA}(\pi_1\Sigma)$.  The set $\cY_{\mc{CA}(\pi_1\Sigma),\cD}$ consists of all presentations for groups $G$ such that word problem in $\cD$ confirms that the canonical homomorphism $G\to\pi_1\Sigma$ has non-abelian kernel.
\end{proof}

\begin{theorem}\label{recursivecirclebundles}
The set $\cB_{-,=}$ and $\cB_{-,\neq}$ are both recursive modulo the word problem.
\end{theorem}
\begin{proof}
A systematic search using Tietze transformations will exhibit the isomorphism type of the centre of a group in $\mc{CAS}$.  In particular, $\mc{CIS}$, the set of central extensions of surface groups by infinite cyclic groups, is recursive in $\mc{CAS}$.    But $\cB_{-,=}\cup \cB_{-,\neq}=\mc{CIS}$. Finally, note that the Euler number of the bundle is exhibited by the abelianization.
\end{proof}

\subsection{Selberg's Lemma}

In this section, we summarize the proof of Selberg's Lemma (in the characteristic zero case), roughly following Alperin \cite{alperin_elementary_1987}.

\begin{theorem}[Selberg's Lemma]
Let $A$ be a finitely generated integral domain of characteristic zero.  Then $GL_n(A)$ has a torsion-free subgroup of finite index.
\end{theorem}

We start with a presentation of $A$ as $A\cong\bZ[\underline{X}]/I$ for $I$ some radical ideal.  Let $K$ be the field of fractions of $A$.  By standard theory, $K$ is an algebraic extension (of degree $d$, say) of the purely transcendental field extension $k=\bQ(\underline{U})$ (where $\underline{U}$ can be identified with a subset of $\underline{X}$).  Let $\{b_1,\ldots,b_d\}$ be a basis for $K$ over $k$.  For each $X_i\in\underline{X}\smallsetminus\underline{U}$, write 
\[
X_i=\sum_{j}\frac{f_{ij}}{g_{ij}} b_j
\]
where $f_{ij},g_{ij}\in \bZ[\underline{U}]$.  Let $B=\bZ[\underline{U}][\frac{1}{\prod_{i,j}g_{ij}}]$.  Then we have a natural embedding
\[
A\into GL_d(B)\subseteq GL_d(k)
\]
and hence a natural embedding
\[
GL_n(A)\into GL_{nd}(B)\into GL_{nd}(k)~. 
\]
Therefore, it suffices to prove Selberg's Lemma for the finitely generated subring $B$ of $k$, a purely transcendental extension of $\bQ$.

\begin{lemma}
If $\alpha\in GL_{nd}(k)$ is of finite order then the trace of $\alpha$ is an integer of absolute value at most $nd$.  Furthermore, it is equal to $nd$ if and only if $\alpha$ is the identity.
\end{lemma}
\begin{proof}
The minimal polynomial of $\alpha$ divides $X^m-1$ for some $m$, so
every eigenvalue of $\alpha$ is a root of unity.  Therefore, the trace
of $\alpha$, which is the sum of the eigenvalues, has absolute value
at most $nd$, with equality only if $\alpha$ is a scalar matrix.  
On the other hand, the trace is a sum of algebraic integers, and hence is itself an algebraic integer.  But the algebraic integers in $k$ are precisely $\bZ$, and so the trace is an integer.
\end{proof}

\begin{proof}[Proof of Selberg's Lemma]
By the above discussion, it suffices to prove the result for a group
$GL_{nd}(B)$ for $B$ some finitely generated subring of $k$.  Because
$B$ is a Jacobson ring \cite[Theorem 4.19]{eisenbud_commutative_1995},
there is a maximal ideal $\mathfrak{m}$ that does not contain $(2nd)!$.  Consider the natural reduction homomorphisms
\[
\eta:GL_{nd}(B)\to GL_{nd}(B/\mathfrak{m})~.
\]
If $\alpha$ is nontrivial but of finite order then, by the previous lemma, $\trace~\eta(\alpha)\neq nd$ and so $\alpha\notin\ker \eta$.
Therefore, $\ker\eta$ is the required torsion-free subgroup of finite index. 
\end{proof}

We can enumerate finite fields, and it is therefore immediate that, given a presentation for $B$, we can search for suitable maximal ideals $\mathfrak{m}_i\subseteq B$.  The computational difficulty, therefore, is to compute $d$ and $B$ given $A$.

\begin{theorem}[Effective Selberg's Lemma]
There is an algorithm that takes as input a presentation for a finitely generated integral domain $A$ and an integer $n$ and outputs an integer $N$ with the property that $GL_n(A)$ has a torsion-free subgroup of index at most $N$.
\end{theorem}
\begin{proof}
Let $A$ be presented as $\bZ[\underline{X}]/I$.  As mentioned above,
the main difficulty is to compute $d$ and a presentation for $B$.
First, the DIMENSION algorithm from \cite{bw:gro} will find
$\underline{U}\subseteq\underline{X}$.  The remaining
$X_i\notin\underline{U}$ are all algebraic, and so a naive search will
eventually find polynomials $h_i$ over $k$ satisfied by the $X_i$.  This provides a spanning set for $K$ over $k$, given by powers of $X_i\notin\underline{U}$.  Now standard linear algebra over $k$ reduces this spanning set to a basis $\{b_1,\ldots,b_d\}$.

A non-deterministic search will express the remaining $X_i$ in terms of the $b_j$, and hence compute $g_{ij}$. One can now write down a presentation for $B$.  Because $B$ is a Jacobson ring, there is a maximal ideal $\mathfrak{m}$ such that $(2nd)!\notin \mathfrak{m}$ as above, and a non-deterministic search will eventually find $\mathfrak{m}$. Now $N=|GL_{nd}(B/\mathfrak{m})|$.
\end{proof}

\subsection{The $PSL_2(\bC)$-representation variety}

The $PSL_2\bC$-representation variety  of a group $\Gamma$ is defined, as a set, by
\[
R(\Gamma)=\Hom(\Gamma,PSL_2\bC)~.
\] 
It can be given the structure of a complex variety by exhibiting
$PSL_2(\bC)$ as an algebraic group over $\bC$.  For completeness, we
do so in the following lemma.
\begin{lemma}
$PSL(2,\bC)$ is an algebraic group over $\bC$.
\end{lemma}
\begin{proof}
We use the image
in $SL(3,\bC)$ of $SL(2,\bC)$, using the second symmetric power of the
standard representation:
\begin{equation}\label{secondsymmetric}
 \rho\co  \left[\begin{array}{cc} a & b\\ c & d\end{array}\right]
  \mapsto 
  \left[
    \begin{array}{ccc}
      a^2 & ab & b^2\\ 
      2ac & ad+bc & 2bd\\ 
      c^2 & cd & d^2
    \end{array}
  \right]
\end{equation}
Note that $\rho$ is well-defined on all of $M_2\bC$.
We claim that $\rho(SL(2,\bC))$ is equal to the following set:
\begin{equation}\label{pslrep}
V=
 \left\{
M = \left[
\begin{array}{ccc}
  e & f & g\\ h & i & j\\ k & l & m
\end{array}
\right]
\ :\ 
\begin{array}{c}
\begin{array}{rr}
  h^2-4ek=0,& j^2-4gm=0,\\ 
  f^2-eg=0, & l^2-km=0,\\
  4fl - hj=0, & 2li-hm-kj = 0 
\end{array}\\
i^2-em-gk-2fl=0,\\
em + gk-2fl =1,\\
\det M = 1
\end{array}
\right\}.
\end{equation}
That $\rho(SL(2,\bC))$ is contained in $V$ is easily checked.  The equations
with right hand side equal to $0$ pick out $\rho(M_2(\bC))$, so every
element of $V$ is $\rho(A)$ for some $A=\left[\begin{array}{cc} a &
    b\\ c & d\end{array}\right]$.  The equation $em + gk-2fl =1$
forces $(\det A)^2=1$, while the equation $\det M = 1$ forces $(\det
A)^3=1$, so $\det A = 1$.  

Finally, it is easy to check that $\ker \rho\cap SL(2,\bC)=\{\pm I\}$.
\end{proof}
The $PSL(2,\bC)$-representation variety of $\Gamma\cong\langle x_1,\ldots,x_n\mid r_1,\ldots r_s\rangle$ is thus an
algebraic subset of $\bC^{9n}$, defined by a collection of $9n$ polynomials coming from \eqref{pslrep} together with $9s$ polynomials coming from the relations. 

Let $V$ be an algebraic component of $R(\Gamma)$.  Each $\gamma\in\Gamma$ naturally defines an evaluation map $\mathrm{ev}_\gamma:V\to PSL_2(\bC)$.  The coordinates of this map are polynomial, and the map $\gamma\mapsto\mathrm{ev}_\gamma$ defines the \emph{tautological representation} $\rho_V:\Gamma\to PSL_2(\bC[V])$. 

\begin{lemma}\label{lem: Tautological representation}
Suppose $\Gamma$ fits into a short exact sequence
\[
1\to\langle z\rangle\to\Gamma\stackrel{p}{\to} F\to 1
\]
where $F$ is a Fuchsian group. There is a component $V$ of $R(\Gamma)$ such that $\ker\rho_V=\langle z\rangle$.
\end{lemma}
\begin{proof}
Let $\sigma:F\to PSL_2(\bC)$ be a Fuchsian representation of $F$, let $\rho=\sigma\circ p$ and let $V$ be an algebraic component of $R(\Gamma)$ containing $\rho$.

The representation $\rho$ has an analytic neighbourhood in $R(\Gamma)$ in which every representation is non-abelian, and hence on which $\mathrm{ev}_z =1$.  This is a polynomial equation, so $\mathrm{ev}_z=1$ on the whole of $V$ and hence $\rho_V(z)=1$.  Therefore, $\langle z\rangle\subseteq\ker\rho_V$.

Conversely, suppose $\gamma\notin\langle z\rangle$.  Then $\rho(\gamma)\neq 1$, so $\mathrm{ev}_\gamma(\rho)\neq 1$ and hence $\rho_V(\gamma)\neq 1$.
\end{proof}

\begin{lemma}
There is an algorithm that takes as input a finite presentation for a group $\Gamma$ and outputs a list of presentations for the polynomial rings $\bC[V_i]$, where $\{V_i\}$ are the irreducible components of $R(\Gamma)$.
\end{lemma}
\begin{proof}
The algorithm starts by using the presentation to compute a set of variables $\underline{X}=\{X_1,\ldots,X_d\}$ and polynomials $f_i(\underline{X})\in\bZ[\underline{X}]$ such that the zero-set of the $f_i$ in $\bC^d$ is the representation variety $R(\Gamma)$.

The next step is to decompose $R(\Gamma)$ into irreducible components $V_i$ by decomposing $(\underline{f})$ into primary ideals $\mathfrak{p}_i$. For each $i$, we then need to compute the radical of $\mathfrak{p}_i$. Both of these steps are performed by the algorithm PRIMDEC \cite[p. 396]{bw:gro}
\end{proof}

\begin{theorem}\label{thm:computingselbergnumbers}
There is an algorithm that takes as input a finite presentation for a
group $\Gamma$ and outputs an integer $N(\Gamma)$ with the following
property: if $\Gamma$ fits into a short exact sequence
\[
1\to C \to\Gamma\to F\to 1
\]
where $F$ is a Fuchsian group and $C$ is infinite cyclic,
then $\Gamma$ has a subgroup of index at most $N(\Gamma)$ that is a central extension of a surface group.
\end{theorem}
\begin{proof}
Use the algorithm of the previous lemma to compute the defining ideal
of $\bC[V_i]$.  For each of these, apply the Effective Selberg's Lemma
to compute an integer $N_i$ such that $\rho_{V_i}(\Gamma)$ has a
torsion-free subgroup of index at most $N_i$, and set $N(\Gamma)$ to
be the maximum of the $N_i$.  By Lemma \ref{lem: Tautological
  representation}, there is an $i$ such that $\ker\rho_{V_i}$ is equal
to $C$.  Thus there is a subgroup of index at most $\max\{N_i\}$ which
is the fundamental group of a circle bundle over a surface, and therefore a
subgroup of index at most $2\max\{N_i\}$ which is a central extension
of a surface group.
\end{proof}

We are now ready to prove the main theorem of this section.

\begin{theorem}\label{thm: SF-}
The sets $\cSF_{-,=}$ and $\cSF_{-,\neq}$ are recursive modulo the word problem.
\end{theorem}
\begin{proof}
We give the proof for $\cSF_{-,=}$; the proof for $\cSF_{-,\neq}$ is
identical.  Let $\cV_N\cB_{-,=}$ denote the class of presentations $P$
with a subgroup of index at most $N(P)$ in $\cB_{-,=}$, where $N(P)$
is the computable function given by Theorem \ref{thm:computingselbergnumbers}.
Since $\cB_{-,=}$ is recursive modulo the word problem (Theorem
\ref{recursivecirclebundles}) and $N$ is a computable function of
$P$, the class
$\cV_N\cB_{-,=}$ is recursive modulo the word problem (Lemma \ref{lem: Virtually recursive modulo the word problem}).

It remains to show that $\cSF_{-,=}$ is recursive in
$\cV_N\cB_{-,=}$.  But $\cSF_{-,=}=\cV_N\cB_{-,=}\cap\cM_3$ and its
complement $\cV_N\cB_{-,=}\smallsetminus \cSF_{-,=}\subseteq
\mc{FT}$.  Since $\cM_3$ and $\mc{FT}$ are both recursively
enumerable (Lemmas \ref{m3re} and \ref{FTRE}), the result follows.
\end{proof}

\section{The hyperbolic case}\label{sec:hyp}

In this section, we prove that $\cH$ is recursive modulo the word problem.  The argument here is based on the argument of \cite{manning:casson}.
\begin{lemma}\label{lem:Hre}
  $\cH$ is recursively enumerable.
\end{lemma}
\begin{proof}
  As noted in the proof of \ref{m3re}, the collection of
  triangulations of closed $3$--manifolds is recursively enumerable.
  From each triangulation, we can obtain a presentation of its
  fundamental group.  The word problem is uniformly solvable in the
  presentations so obtained, by Proposition \ref{hempelcor}.  The main
  theorem of \cite{manning:casson} then implies that we can distinguish
  the hyperbolic manifolds in this list from the non-hyperbolic ones.  

  Performing Tietze transformations on the presentations coming from
  triangulations of hyperbolic manifolds yields the desired enumeration.
\end{proof}
\begin{remark}
  Here is an alternate way to divide closed triangulated $3$--manifolds into
  hyperbolic and non-hyperbolic, without reference to the word
  problem.  (See \cite[Section 1.4]{BBBMP} or \cite{Jacolectures} for
  more detail.)
  Starting with a triangulated $3$--manifold $M$, use normal
  surface theory to find a collection of $2$--spheres decomposing $M$
  into irreducible $3$--manifolds $N_1,\ldots N_m$ \cite[Section 7]{jactol:decomp}.  If
  any are nonseparating, then $M$ is not hyperbolic.  Assuming the
  spheres are all separating, one checks whether each $N_i$ is
  homeomorphic to a $3$--sphere \cite{rubinstein:s3,thompson:s3}.  If
  every $N_i\cong S^3$, then $M\cong S^3$, and is not hyperbolic.  If
  $2$ or more $N_i$ are not $S^3$, then $M$ is a nontrivial connect
  sum, and is not hyperbolic.  So assume that exactly one
  $N_i\not\cong S^3$.  We then split $M$ into simple and
  characteristic submanifolds \cite[Section 8]{jactol:decomp}.  If
  this decomposition is nontrivial, then $M$ is not hyperbolic.  If
  this decomposition is trivial, the Geometrization Theorem implies
  that $M$ is either hyperbolic or Seifert fibred.  Tao Li in
  \cite{TaoLi06} gives an algorithm which determines whether or not
  $M$ is Seifert fibred.
\end{remark}

Lemma \ref{lem:Hre} implies that for any collection $\mc{D}$ of
presentations with uniformly 
solvable word problem, we may take $\mc{X}_{\cH,\cD} = \mc{\cH}$.  It
remains to describe $\mc{Y}_{\cH,\cD}$.

Here is a sketch:  We begin with some enumeration of $\mc{U}$.  
Let $G$ be the group presented by a presentation $\mc{P}$.  
We construct a finite list of candidate representations of $G$ into
$SL(2,\bC)$ with the property
that if $G$ is a closed hyperbolic 3--manifold group, then $\rho_i$
must be discrete and faithful for some $i$ (Lemma \ref{l:findreps}).
We then use the uniform word problem for $\cD$ as part of a procedure
to certify  that none of these representations is a discrete faithful
representation of a closed hyperbolic $3$--manifold group.
(This procedure may produce a ``fake certificate'' for presentations in
$\mc{H}\smallsetminus \mc{D}$, but the behaviour of our algorithm on
$\mc{U}\smallsetminus\mc{D}$ is irrelevant.)

\begin{definition}
Say that a representation $\rho\co G\to SL(2,\bC)$ is \emph{DFIL} if it is
discrete and $\trace(\rho(g))\in \bC\smallsetminus [-2,2]$ for all $g\neq
1$.  (DFIL stands for ``\underline{d}iscrete and \underline{f}aithful
with \underline{i}rreducible, all-\underline{l}oxodromic image''.)
\end{definition}

\subsection{Finding the representations}
The algebraic closure of the rationals
is denoted $\Qbar$.  Note that an element of $\Qbar$ can be
represented by a finite amount of data, 
exact arithmetic can be done in
$\Qbar$, and that inequalities between elements of $\Qbar\cap \bR$ can
be decided.  (See \cite[Section 2.1]{manning:casson} for more details.)

A representation $\rho\co G\to SL(2,\bC)$ is \emph{rigid} if any nearby
representation is conjugate.  An irreducible rigid representation projects to an
isolated point of the $SL(2,\bC)$--character variety of $G$. 

Say a representation $\rho\co G\to SL(2,\bC)$ is
\emph{algebraic} if its image lies in $SL(2,\Qbar)$.  Such a
representation can be specified using a finite amount of data.  A
\emph{representation of} $\mc{P} = \langle x_1,\ldots, x_n \mid
r_1(x_1,\ldots,x_n),\ldots,r_s(x_1,\ldots,x_n)\rangle$ is a list of
$n$ matrices $A_1,\ldots,A_n$ so that $r_i(A_1,\ldots,A_n) = I$ for
each $i$ between $1$ and $s$.  If $\mc{P}$ is a presentation for $G$,
then a representation of $G$ determines one for $\mc{P}$, and vice
versa.  We will blur the distinction, but this leads to
no errors:  In the algorithms we discuss, we always operate on a
particular presentation, and never on an abstract group.

The following is a restatement of Lemma 2.1 of \cite{manning:casson}, with
the unused assumption that $G$ is a $3$--manifold group removed.
\begin{lemma}\label{l:findreps}
  There is an algorithm which takes as input a finite presentation
  $\mc{P}$ of a group $G$
  and outputs a finite list of rigid algebraic representations of $\mc{P}$.
  If $G$ is the fundamental group of a
  closed orientable hyperbolic $3$--manifold, then some representation
  on this list is DFIL.
\end{lemma}
\begin{proof}
  See the proof of \cite[Lemma 2.1]{manning:casson}.  The idea is to
  use computational algebraic geometry to construct a sub-variety of
  the representation variety whose isolated points are
  rigid representations of $G$, and so that any irreducible rigid
  representation is conjugate to one of these isolated points.
\end{proof}

\subsection{Rejecting a representation which is not DFIL}\label{s:ndfl}
Section 3 of \cite{manning:casson} contains a proof that there is an
algorithm which will stop if a given algebraic representation
of a closed $3$--manifold group with solvable word problem
fails to be DFIL.  In fact, the proof does not use that the group is a
$3$--manifold group.  We therefore have
\begin{theorem}\label{t:ndfil}
  There is an algorithm which takes as input a finite presentation
  $\mc{P}$, a solution to the word problem in $\mc{P}$, and an
  algebraic representation $\rho$ of $\mc{P}$, and terminates if and
  only if $\rho$ is not DFIL.
\end{theorem}
\begin{proof}
  See the proof of \cite[Theorem 3.6]{manning:casson}.  The key to the
  proof is that a representation which is not DFIL must satisfy one of
  the following:
  \begin{enumerate}
  \item $\rho$ has nontrivial kernel.
  \item $\rho$ is a reducible representation; i.e., some direction in
    $\bC P^1$ is fixed by all the matrices
    $\rho(x_i)$, where $x_i$ is a generator from $\mc{P}$.
  \item $\rho(g)$ is a nontrivial parabolic or elliptic for some $g\in
    G$.
  \item $\rho(G)$ fails to satisfy the Margulis lemma.  Put another
    way, there are a pair of matrices in $\rho(G)$ which move some point in
    $\bH^3$ a very small distance, but whose commutator is nontrivial.
  \end{enumerate}
  The arguments in Section 3 of \cite{manning:casson} establish that
  the failure of any of these conditions is algorithmically
  detectable, given a solution to the word problem in $\mc{P}$, and
  the ability to do exact computations in the codomain of $\rho$.
\end{proof}

\subsection{$\cH$ is recursive in $\mc{D}$}

\begin{theorem}\label{thm: hyp}
The set $\cH$ is recursive modulo the word problem.
\end{theorem}
\begin{proof}
  Let $\mc{D}$ be any class of presentations with uniformly solvable
  word problem.  We must find recursively enumerable sets
  $X_{\mc{H},\mc{D}}$ and $Y_{\mc{H},\mc{D}}$ whose intersections with
  $\cD$ partition $\cD$ into $\cD\cap\cH$ and its complement.

  Lemma \ref{lem:Hre} implies we may take $X_{\mc{H},\mc{D}}=\cH$.

  We now describe the Turing machine enumerating $Y_{\mc{H},\mc{D}}$.
  Enumerate the elements of $\mc{U}$.  For each presentation in
  $\mc{U}$, compute the list of representations guaranteed by Lemma
  \ref{l:findreps}.  In parallel, run the algorithm from Theorem
  \ref{t:ndfil} on each such representation produced, trying to
  certify the representation is not DFIL.  If this algorithm
  terminates for all the representations associated to a particular
  presentation, then output that presentation.  This procedure will
  eventually produce every element of $\mc{D}\smallsetminus\mc{H}$, and will
  produce no elements of $\mc{D}\cap\mc{H}$, so the theorem is proved.
\end{proof}

\small
\def\cprime{$'$}

\end{document}